\documentclass[11pt]{amsart}
\usepackage{amsmath,amssymb,latexsym,esint,cite,mathrsfs}
\usepackage{verbatim,wasysym}
\usepackage[left=2.5cm,right=2.5cm,top=2.7cm,bottom=2.7cm]{geometry}
\usepackage{tikz,enumitem,graphicx, subfig, microtype, color}
\usepackage{epic,eepic}

\usepackage[colorlinks=true,urlcolor=blue, citecolor=red,linkcolor=blue,
linktocpage,pdfpagelabels, bookmarksnumbered,bookmarksopen]{hyperref}
\usepackage[hyperpageref]{backref}
\usepackage[english]{babel}
\allowdisplaybreaks
\numberwithin{equation}{section}
\newtheorem{theorem}{Theorem}[section]
\newtheorem{proposition}[theorem]{Proposition}
\newtheorem{corollary}[theorem]{Corollary}
\newtheorem{lemma}[theorem]{Lemma}
\theoremstyle{definition}

\newtheorem{definition}[theorem]{Definition}
\newtheorem{remark}[theorem]{Remark}

\newcommand{\R}{\mathbb{R}}

\newcommand{\supp}{{\rm supp}{\hspace{.05cm}}}
\begin{document}

\title
 [Critical Gauged Schr\"{o}dinger equations in $\R^2$]
 {Critical gauged Schr\"{o}dinger equations \\
 in $\R^2$ with vanishing potentials}

\author[L.\ Shen]{Liejun Shen}
\author[M.\ Squassina]{Marco Squassina}
\author[M.\ Yang]{Minbo Yang}

\address{Liejun Shen, Minbo Yang, \newline\indent Department of Mathematics, Zhejiang Normal University, \newline\indent
	Jinhua, Zhejiang, 321004, People's Republic of China}
\email{liejunshen@sina.com, mbyang@zjnu.edu.cn.}

\address{Marco Squassina, \newline\indent
	Dipartimento di Matematica e Fisica \newline\indent
	Universit\`a Cattolica del Sacro Cuore, \newline\indent
	Via della Garzetta 48, 25133, Brescia, Italy}
\email{marco.squassina@unicatt.it}


\thanks{Minbo Yang was partially supported by NSFC (11971436, 12011530199) and ZJNSF(LZ22A010001, LD19A010001). Marco  Squassina  is  member  of  the  Gruppo  Nazionale  per  l'Analisi  Matematica,  la Probabilita  e  le  loro  Applicazioni  (GNAMPA)  of  the  Istituto  Nazionale  di  Alta  Matematica  (INdAM)}

\begin{abstract}
We study a class of gauged nonlinear Schr\"{o}dinger equations in the plane
\[
\left\{ \begin{gathered}
  -\Delta u+V(|x|) u+\lambda\bigg(\int_{|x|}^\infty \frac{h_u(s)}{s}u^2(s)ds+\frac{h_u^2(|x|)}{|x|^2} \bigg)u=
  K(|x|)f(u)+\mu g(|x|)|u|^{q-2}u, \hfill \\
 u(x)=u(|x|) ~\text{in}~\R^2, \hfill \\
\end{gathered}  \right.
\]
where $
h_u(s)=\int_0^s\frac{r}{2}u^2(r)dr
$, $\lambda,\mu>0$ are constants,
$V(|x|)$ and $K(|x|)$ are continuous
functions vanishing at infinity. Assume that $f$ is of critical exponential growth and $g(x)=g(|x|)$ satisfies some technical assumptions with $1\leq q<2$, we obtain
 the existence of two nontrivial solutions via the Mountain-Pass theorem and Ekeland's variational principle.
  Moreover, with the help of
   the genus theory, we prove the existence of infinitely many solutions
if $f$ in addition is odd.
\end{abstract}


%

\maketitle

\section{Introduction and main results}

\setcounter{equation}{0}

\subsection{General overview}
In this paper we consider the existence and multiplicity of nontrivial solutions for a
gauged nonlinear
 Schr\"{o}dinger equation with vanishing potentials and critical exponential growth
\begin{equation}\label{mainequation1}
\left\{ \begin{gathered}
  -\Delta u+V(|x|) u+\lambda\bigg(\int_{|x|}^\infty \frac{h_u(s)}{s}u^2(s)ds+\frac{h_u^2(|x|)}{|x|^2} \bigg)u=
  K(|x|)f(u)+\mu g(|x|)|u|^{q-2}u, \hfill \\
 u(x)=u(|x|)~\text{in}~\R^2, \hfill \\
\end{gathered}  \right.
\end{equation}
where $h_u(s)=\int_0^s\frac{r}{2}u^2(r)dr$, $\lambda,\mu>0$ are constants,
$V(|x|)$ and $K(|x|)$ are continuous
functions vanishing at infinity,  $f$ is of critical exponential growth and $g(x)=g(|x|)$ satisfies some technical assumptions with $1\leq q<2$.
The study of equation \eqref{mainequation1}
is mainly motivated by the  Chern-Simons-Schr\"{o}dinger system introduced in \cite{Jackiw1,Jackiw2}
\begin{equation}\label{CSS1}
\begin{gathered}
iD_0\phi+(D_1D_1+D_2D_2)\phi=-\varrho(\phi),~
\partial_0A_1-\partial_1A_0=-\text{Im}(\overline{\phi} D_2\phi),\hfill\\
\partial_0A_2-\partial_2A_0=\text{Im}(\overline{\phi} D_1\phi),~
\partial_1A_2-\partial_2A_1=-\frac{1}{2}|\phi|^2, \hfill\\
\end{gathered}
\end{equation}
This system consists of the nonlinear Schr\"{o}dinger equation augmented by the
gauge field $A_j:\R^{1+2}\to\R$, where $i$ denotes the imaginary unit,
$\partial_0=\partial/\partial t$, $\partial_1=\partial/\partial x_1$, $\partial_2=\partial/\partial x_2$
for $(t,x_1,x_2)\in\R^{1+2}$, $\phi:\R^{1+2}\to \mathbb{C}$ is the complex scalar field and
$D_j=\partial_j iA_j$ is the covariant derivative for $j=0, 1, 2$.
For each $C_0^\infty(\R^{1+2})$ function $\chi$, under
the following gauge transformation
\[
\phi\to \phi e^{i\chi},~
A_j\to A_j-\partial_j\chi,
\]
system  \eqref{CSS1} is invariant because of the Chern-Simons theory \cite{Dunne}.

To study the existence of standing waves of system \eqref{CSS1},
Byeon-Huh-Seok \cite{Byeon1}
investigated the existence of solutions of type
\begin{equation}\label{CSS2}
\begin{gathered}
\phi(t,x)=u(|x|)e^{i\omega t},~ A_0(t,x)=k(|x|), \hfill\\
A_1(t,x)=\frac{x_2}{|x|^2}h_u(|x|),~  A_2(t,x)=\frac{x_1}{|x|^2}h_u(|x|),\hfill\\
\end{gathered}
\end{equation}
where $\omega>0$ denotes the frequency and $u,k,h$ are real value functions depending only
on $|x|$. Note that \eqref{CSS2} satisfies the Coulomb gauge condition with $\chi=ct+n\pi$, where
$n$ is an integer and $c$ is a real constant.
Indeed, inserting \eqref{CSS2} into \eqref{CSS1},
it can be reduced to the following semilinear elliptic equation
\begin{equation}\label{BHS}
 -\Delta u+ (\omega+\zeta) u+ \bigg(\int_{|x|}^\infty \frac{h_u(s)}{s}u^2(s)ds+\frac{h_u^2(|x|)}{|x|^2} \bigg)u=
  \varrho(u)  ~ \text{in}~ \mathbb{R}^2,
\end{equation}
where $\varrho(u)=\overline{\lambda} |u|^{p-2}u$ with $\overline{\lambda}>0$,
$h(s)=\int_0^s\frac{r}{2}u^2(r)dr$,
and $\zeta\in\R$ stands for an integration constant of $A_0$ which takes the form
\[
A_0(r)=\zeta+\int_{r}^\infty \frac{h_u(s)}{s}u^2(s)ds.
\]
Since the constant $\omega+\zeta$ is a gauge invariant of the stationary solutions,
one can take $\zeta=0$ in \eqref{BHS} for simplicity in what follows
and hence
\[
\lim_{|x|\to\infty}A_0(x)=0
\]
which was assumed in \cite{Berge,Jackiw1,Pomponio}.
If $\overline{u}$ solves \eqref{BHS}, inspired by \cite{Deng}, $u= \lambda^{\frac{1}{p-2}}\overline{u}$
  satisfies
\begin{equation}\label{BHS1}
 -\Delta u+  \omega u+\lambda \bigg(\int_{|x|}^\infty \frac{h_u(s)}{s}u^2(s)ds+\frac{h_u^2(|x|)}{|x|^2} \bigg)u=
  |u|^{p-2}u ~ \text{in}~ \mathbb{R}^2,
\end{equation}
where $\lambda=\overline{\lambda}^{-\frac{4}{p-2}}$.
Over the past several decades, equation \eqref{BHS1} has attracted
a lot of interest due to the appearance of the nonlocal
Chern-Simons term
\begin{equation}\label{CSt}
\int_{|x|}^\infty \frac{h_u(s)}{s}u^2(s)ds+\frac{h_u^2(|x|)}{|x|^2},
\end{equation}
which indicates that equation \eqref{BHS1} is not a pointwise identity any longer.
Byeon-Huh-Seok \cite{Byeon1} established the existence of ground state solutions
for every $p>4$ by a suitable constraint minimization procedure, existence and nonexistence of nontrivial
solutions depending on $\lambda>0$ for $p = 4$, and the existence of minimizers under $L^2$-constraint
for every $p\in(2,4)$. In \cite{Pomponio}, the authors investigated that there exists a sharp constant $\omega_0>0$
such that the corresponding variational functional to equation \eqref{BHS1} is bounded from below if $\omega\geq
\omega_0$ and not bounded from below for all $\omega\in(0,\omega_0)$ with $p\in(2,4)$.
By replacing $|u|^{p-2}u$ with a general nonlinearity $f(u)$ in \eqref{BHS1},
 authors in \cite{Cunha} established the multiple results when $f(u)$ is a Berestycki-Gallou\"{e}t-Kavian
type nonlinearity \cite{Berestycki3} and it is the planar version of the well-known
Berestycki-Lions type nonlinearity \cite{Berestycki2,Berestycki3}. Besides,
there are also some other interesting and meaningful
research works on equation \eqref{BHS1} and involving general classes of nonlinearities,
we refer the reader to \cite{Huh1,Pomponio,Pomponio1,Byeon2,2Jiang,
Deng} and the references therein.

\subsection{Handling the planar case}
Let's point out here that the spatial dimension of equations \eqref{mainequation1} and \eqref{BHS1},
is two, thereby the case is special and quite delicate.
Roughly speaking, the Sobolev embedding theorem
ensures $H_0^1(\Omega)\hookrightarrow L^s(\Omega)$ with
$s\in[1,\infty)$ for each bounded domain $\Omega\subset\R^2$,
but $H_0^1(\Omega)\not\hookrightarrow L^\infty(\Omega)$.
Hence,
to overcome the obstacle in the limiting case,
the Trudinger-Moser inequality \cite{PSI,TNS,MJ} can be treated as a substitute of
the Sobolev inequality since it establishes the following sharp maximal exponential integrability for
functions in $H_{0}^{1}(\Omega)$:
\begin{equation}\label{TM}
\sup\limits_{u\in H_{0}^{1}(\Omega):\|\nabla u\|_{L^2
(\Omega)}\leq1}\int_{\Omega}e^{\alpha u^{2}}dx\leq C|\Omega|~\mbox{if}~\alpha\leq4\pi,
\end{equation}
where $C>0$ depends only on $\alpha$, and $|\Omega|$ denotes
the Lebesgue measure of $\Omega$. Subsequently,
this inequality was generalized by P. L. Lions in \cite{LPL}:
Let $\{u_{n}\}$ be a sequence of functions in $H_{0}^{1}(\Omega)$ with $\|\nabla u_{n}\|_{L^2(\Omega)}=1$ such that
$u_{n}\rightharpoonup u_{0}$ weakly in $H_{0}^{1}(\Omega)$, then for all $p<\frac{1}{(1-\|\nabla u_{0}\|_{2}^{2})}$,
there holds
$$
\limsup\limits_{n\rightarrow\infty}\int_{\Omega}e^{4\pi pu_{n}^{2}}dx<+\infty.
$$
Inspired by the Trudinger-Moser type inequality, we
 say that a function $f(s)$ is of \emph{critical exponential growth} if there exists
a constant $\alpha_{0}>0$ such that
\begin{equation}\label{definition}
\lim\limits_{|s|\rightarrow+\infty}
\frac{|f(s)|}{e^{\alpha s^{2}}}=
\left\{
  \begin{array}{ll}
    0, & \forall
\alpha>\alpha_{0}, \\
    +\infty, &\forall \alpha<\alpha_{0}.
  \end{array}
\right.
\end{equation}
This definition was introduced by Adimurthi and Yadava \cite{AYA},
see also de Figueiredo, Miyagaki and Ruf  \cite{Figueiredo} for example.

Unfortunately,
the supremum in \eqref{TM} becomes infinite for domains $\Omega$ with $|\Omega|=\infty$,
and therefore the Trudinger-Moser inequality is not available for the unbounded domains.
As to the whole space $\R^2$, the author in \cite{Bezerra1}
established the following version of the Trudinger-Moser inequality
(see also \cite{Cao} for example):
\[
e^{\alpha u^2}-1\in L^2(\R^2),~\forall \alpha>0~\text{and}~u\in H^1(\R^2).
\]
Moreover, for every $u\in H^1(\R^2)$ with $\|u\|_{L^2(\R^2)}\leq M<+\infty$, there exists a positive constant $C=C(M,\alpha)$
such that
\[
\sup\limits_{u\in H^{1}(\R^2):\|\nabla u\|_{L^2
(\R^2)}\leq1}\int_{\R^2}\big(e^{\alpha u^{2}}-1\big)dx\leq C ~\mbox{if}~\alpha<4\pi.
\]
Concerning some other generalizations, extensions and applications
of the Trudinger-Moser inequalities for bounded and unbounded domains,
we refer to \cite{Figueiredo} and its references therein.
It should be noted that the inequality by Cao \cite{Cao} holds only strictly for $\alpha<4\pi$, i.e. with subcritical growth. For the sharp case, based on symmetrization and blow-up analysis, Ruf  \cite{Ruf} , Li and Ruf \cite{LR} proved that
\begin{equation*}
\sup_{u\in W_{0}^{1,N}(\mathbb{R}^{N}),\|u\|_{L^{N}}^{N}+\|\nabla u\|_{L^{N}}^{N}\leq1}\int_{\mathbb{R}^{N}}\left(e^{\alpha|u|^{\frac{N}{N-1}}}
-\sum_{k=0}^{N-2}\frac{\alpha^{k}|u|^{kN/(N-1)}}{k!}\right)dx< \infty,\mbox{ if }\alpha\leq\alpha_{N},
\end{equation*}
by replacing the $L^{N}$ norm of $\nabla u$ in the supremum with the standard Sobolev norm.
This inequality was improved by Souza and do \'{O}\cite{DSD} for $N=2$.
Let $(u_{n})$ be in $E$ with $\|u_{n}\|=1$ and suppose that $u_{n}\rightharpoonup u_{0}$ in $E$. Then for all $0<p<\frac{4\pi}{1-\|u_{0}\|^{2}}$, the authors  proved that
$$
\sup_{n}\int_{\mathbb{R}^{2}}(e^{pu_{n}^{2}}-1)dx<\infty.
$$
We refer the readers to the references in the bibliography of this paper for more information about the progress on the elliptic equations with critical exponential growth.

since the problem was set in $\R^2$, it is quite natural to study the existence results for the gauged nonlinear Schr\"{o}dinger equations with
critical growth in the sense of Trudinger-Moser inequality.
The aim of this paper is to continue the investigation of
the gauged Schr\"{o}dinger equations by considering the potentials $V(|x|)$
and $K(|x|)$ (replacing $f(x,u)$ with $K(|x|)f(u)$), which can be
singular at the origin and vanishing at infinity.

\subsection{Assumptions and functional setting}
We impose the hypotheses on $V(|x|)$ and $K(|x|)$
as follow:
\begin{itemize}
  \item  [$(V_0)$] $V\in C(0,\infty)$, $V (r) > 0$ for $r > 0$
  and there exist $a_0>-2$ and $\frac{2}{3} <a<2$ such
that
\[
\limsup_{r\to0^+}\frac{V (r)}{r^{a_0}}<\infty~\text{and}~
\liminf_{r\to+\infty}\frac{V (r)}{r^{a}}>0;
\]
\end{itemize}

\begin{itemize}
  \item  [$(K_0)$] $K\in C(0,\infty)$, $K(r) > 0$ for $r > 0$
  and there exist $b_0>-2$ and $b<a$ such
that
\[
\limsup_{r\to0^+}\frac{K (r)}{r^{b_0}}<\infty~\text{and}~
\limsup_{r\to+\infty}\frac{K (r)}{r^{b}}<\infty.
\]
\end{itemize}
In the sequel, we mean that
$(V,K)\in  \mathcal{K}$ if the continuous functions $V(|x|)$ and $K(|x|)$ satisfy
$(V_0)$ and $(K_0)$, respectively.
Let's mention here that the similar conditions
were presented in \cite{Albuquerque1,Albuquerque2}. However, compared with \cite{Albuquerque2}
the proof for  the nonlocal
Chern-Simons case \eqref{CSt} is much more complicated  and
technical.  Moreover, as far as we know, Ji and Fang
\cite{Ji} had considered the existence and multiplicity
of nontrivial solutions
to the nonhomogeneous Chern-Simons-Schr\"{o}dinger system with strictly positive potential.

 In this work, we suppose that the nonlinearity $f(t)$
is of \emph{critical exponential growth} for $f$  and we also assume that $f$ and $g$ satisfy the following conditions
\begin{itemize}
\item  [$(f_1)$] \emph{$f\in C(\R,\R)$ with $f(t)\equiv0$ for all $t\leq0$ and $f(t)=o(t)$ as $t\to0^+$;}
\item  [$(f_2)$] \emph{$f(t)t-6F(t)\geq 0$ for each $t\in\R$ and $\lim_{t\to+\infty}F(t)/t^6=+\infty$,
where $F(t)= \int_0^tf(s)ds$;}
\item [$(f_3)$] \emph{there exist two constants $p>6$ and $\kappa>0$ such that
 $F(t)\geq \kappa t^p$ for all $t\in[0,1]$;}
\end{itemize}
\begin{itemize}
\item [$(g)$] \emph{$0\leq g(x)=g(|x|)\in L^\infty_{\emph{\text{loc}}}(\R^2)$ and there exist two constants $1\leq q<2$
with $\sigma<q-2$ such that
$\limsup_{r\to+\infty}g(r)/[r^\sigma V^{q/2}(r)]<+\infty.$}
\end{itemize}

Let's denote by $H^1(\R^2)$ the
usual Sobolev space equipped with the usual inner product and norm
\[
(u,v)_{H^1(\R^2)}=\int_{\R^2}\big[\nabla u\nabla v+ uv\big]dx~ \text{and}~
\|u\|_{H^1(\R^2)}= (u,u)^{1/2}_{H^1(\R^2)} ,~ \forall u,v\in H^1(\R^2).
\]
Let
$H_r^1(\R^2)=\{u\in H^1(\R^2):u(x)=u(|x|)\}$ be the subspace endowed with the previous inner product and norm.
For each $\Omega\subset\R^2$, we shall exploit $L^m(\Omega)$ to stand for
the usual Lebesgue space with the standard norm $\|\cdot\|_{L^m(\Omega)}$.
In particular, if $\Omega=\R^2$, instead of $\|\cdot\|_{L^m(\R^2)}$, we'll use
$|\cdot|_m$ for simplicity. Given a constant $s\in[1,+\infty)$,
 as in \cite{Albuquerque1,Albuquerque2},
we introduce the weighted Lebesgue functions $L_V^2(\R^N)$ and
 $L_K^s(\R^N)$ as follows
\[
L_V^2(\R^2)\triangleq\bigg\{u:\R^2\to\R
\big|u~\text{is Lebesgue measurable and}
 \int_{\R^2}V(|x|)|u|^2dx<\infty\bigg\}
\]
and
\[
L_K^s(\R^2)\triangleq\bigg\{u:\R^2\to\R
\big|u~\text{is Lebesgue measurable and}
 \int_{\R^2}K(|x|)|u|^sdx<\infty\bigg\}
\]
respectively, whose norms are defined by
\[
|u|_{V,2}=\bigg(\int_{\R^2}V(|x|)|u|^2dx\bigg)^{\frac{1}{2}}~\text{and}~
|u|_{K,s}=\bigg(\int_{\R^2}K(|x|)|u|^sdx\bigg)^{\frac{1}{s}}.
\]
We also define the functional space
\[
X\triangleq\bigg\{u\in L^2_{\text{loc}}(\R^2)
\big||\nabla u|\in L^2 (\R^2)~\text{and}
 \int_{\R^2}V(|x|)|u|^2dx<\infty\bigg\}
\]
equipped with the norm $\|u\|=(u,u)^{1/2}$ induced by the inner product
\[
(u,v) =\int_{\R^2}\big[\nabla u\nabla v+ V(|x|)uv\big]dx,~ \forall u,v\in X.
\]
One can verify that $(X,\|\cdot\|)$ is a Hilbert space.
Moreover, it's clear that $X_r\triangleq\{u\in X:u(x)=u(|x|)\}$ is closed
in $X$ with respect to the topology corresponding to $\|\cdot\|$
and therefore it is a Hilbert space itself.

Let $C_{0,r}^\infty(\R^2)$ be the set of radially smooth functions with compact support,
then $(X_r,\|\cdot\|)$ is the closure of it. Therefore, we say that
$u:\R^2\to\R$ is a (radial-weak)
solution of equation \eqref{mainequation1} provided that
$u\in X_r$ and it holds the equality
 \begin{align}\label{solution}
 \nonumber 0 &=\int_{\R^2} \big[\nabla u \nabla v+ V(|x|)uv\big]dx+ \lambda \int_{\R^2}\frac{u^2}{|x|^2}
  \bigg(\int_{0}^{|x|}\frac{r}{2}u^2(r)dr\bigg)\bigg(\int_{0}^{|x|}ru(r)v(r)dr\bigg)dx \\
   & \ \
   +\lambda\int_{\R^2}\frac{uv}{|x|^2}\bigg(\int_{0}^{|x|}\frac{r}{2}u^2(r)dr\bigg)^2dx
   - \int_{\R^2}K(|x|)f(u)vdx-\mu\int_{\R^2}g(|x|)|u|^{q-2}uvdx,
\end{align}
for all $v\in C_{0,r}^\infty(\R^2)$, where the Fubini's theorem is used.

\subsection{The main results}
Now, we can state the main results as follows.

\begin{theorem}\label{maintheorem1}
Suppose that $(V,K)\in \mathcal{K}$, $f$ satisfies \eqref{definition} and
 $(f_1)$-$(f_3)$, and $(g)$ hold. If the constant $\kappa>0$ given by $(f_3)$
 satisfies $\kappa\geq\kappa^*$, where
 \[
\kappa^*\triangleq\max\bigg\{\kappa_1,\frac{2\kappa_1}{p}
\bigg[\frac{3\alpha_0 \kappa_1 (p-2)\|K\|_{L^1(B_{1/2}(0))}}{p\pi(1+\frac{b_0}{2})}\bigg]^{\frac{p-2}{2}}\bigg\}
  ~\text{\emph{with}}~\kappa_1=\frac{(16+\lambda)\pi+16\|V\|_{L^1(B_1(0))}}{32\|K\|_{L^1(B_{1/2}(0))}},
 \]
then there is a constant $\mu_*>0$ such that equation
\eqref{mainequation1} admits at least two nontrivial solutions
 with radial symmetry for any $\lambda>0$ and $\mu\in(0,\mu_*)$.
\end{theorem}

\begin{remark}
 It should be pointed out that the solutions obtained in Theorem \ref{maintheorem1}
(and in Corollary \ref{corollary} and Theorems \ref{maintheorem2}-\ref{maintheorem3} below)
cannot be recognized as belonging to $X$ since the classical Palais' Principle of
Symmetric Criticality doesn't apply due to the fact that
 the energy functional
$J$ could be not differentiable, not even well-defined, on the whole space $X$.
 We expect to verify these solutions are unnecessarily radially symmetric,
but we postpone this question to a future work.
\end{remark}

By Theorem \ref{mainequation1},
we can immediately obtain the following result.

\begin{corollary}\label{corollary}
Suppose that $(V,K)\in \mathcal{K}$, $f$ satisfies \eqref{definition},
 $(f_1)$-$(f_3)$, and $(g)$ hold. Then, for some sufficiently
 large $\kappa>0$ and small $\mu>0$, equation
\eqref{mainequation1} has two nontrivial solutions for all $\lambda>0$.
\end{corollary}

For the second existence result,
we replace $(f_3)$
by the following hypotheses on $f$ :
\begin{itemize}
  \item [$(f_4)$] \emph{there exist constants $t_0>0$, $M_0>0$ and $\vartheta\in(0,1]$ such that
  $$0<t^\vartheta F(t)\leq M_0f(t),  \forall~ t\geq t_0;$$}
  \item [$(f_5)$] \emph{$\liminf_{t\to+\infty}F(t)/e^{\alpha_0 t^2}\triangleq \beta_0>0$.}
\end{itemize}

\begin{theorem}\label{maintheorem2}
Suppose that $(V,K)\in \mathcal{K}$, $f$ satisfies \eqref{definition}, $(f_1)-(f_2)$ and
$(f_4)-(f_5)$, and (g) hold. If we suppose that $\liminf_{r\to0^+}K(r)/r^{(b_0-22)/{12}}>0$,
  there is a constant $\mu_{**}>0$ such that equation
\eqref{mainequation1} has at least two nontrivial solutions for any $\lambda>0$
and $\mu\in(0,\mu_{**})$.
\end{theorem}

Finally, by applying the genus theory, we investigate the existence of infinitely many solutions for
equation \eqref{mainequation1}.
In fact, we are able to prove a further result.

\begin{theorem}\label{maintheorem3}
Under the assumptions in Theorems \ref{maintheorem1}, or \ref{maintheorem2},
if $f(t)=-f(-t)$ for each $t\in\R$,
then there exists a constant $\mu_{**}^*>0$ such that equation
\eqref{mainequation1} has infinitely many solutions for any $\lambda>0$ and $\mu\in(0,\mu_{**}^*)$.
\end{theorem}

We need to point out that,
different from the previous literatures
\cite{Byeon1,Pomponio,Pomponio1,Byeon2,2Jiang,
Deng,Azzollini}, one cannot easily verify that the functional $c(u)$
introduced in Section 2 below
is well-defined for every $u\in X_r$. In fact, it's obvious to conclude that $0\leq c(u)<+\infty$
for each $u\in L^2(\R^2)\cap L^4(\R^2)$, however, it seems that
$X\not\hookrightarrow L^2(\R^2)\cap L^4(\R^2)$.
As a consequence, this difficulty
prevents us considering equation \eqref{mainequation1} in a standard way.
\\\\
\textbf{Organization of the paper.}
The paper is organized as follow. In Section 2, we'll introduce some
useful preliminaries which can be exploited later on.
In Section 3, by employing the well-known mountain-pass theorem and Ekeland's variational principle, we
search for two different nontrivial
  solutions for equation \eqref{mainequation1} in Theorems \ref{maintheorem1} and \ref{maintheorem2}.
  At last,
we are devoted to establishing the existence of infinitely many solutions for equation
  \eqref{mainequation1} by applying the genus theory in Section 4.
 \\\\
 \textbf{Notations.} Throughout this paper we shall denote by $C$ and $C_i$ ($i\in{\mathbb N}$) for various positive constants
whose exact value may change from lines to lines but are not essential to the analysis of the problem.
 We exploit $``\to"$ and $``\rightharpoonup"$ to denote the strong and weak convergence
in the related function spaces, respectively. For any $\rho>0$ and each $x\in \R^2$, $B_\rho(x)$ denotes the ball of radius
$\rho$ centered at $x$, that is, $B_\rho(x):=\{y\in \R^2:|y-x|<\rho\}$.

Let $(E,\|\cdot\|_E)$ be a Banach space with its dual space $(E^{*},\|\cdot\|_{*})$, and $\Phi$ be its functional on $E$.
The Palais-Smale sequence at level $c\in\R$ ($(PS)_c$ sequence in short) corresponding to $\Phi$ means that $\Phi(x_n)\to c$
and $\Phi^{\prime}(x_n)\to 0$ as $n\to\infty$, where $\{x_n\}\subset E$. If for any $(PS)_c$ sequence $\{x_n\}$ in $E$,
there exists a subsequence $\{x_{n_{k}}\}$ such that $x_{n_{k}}\to x_0$ in $E$ for some $x_0\in E$, then we say that the
functional $\Phi$ satisfies the so called $(PS)_c$ condition.

\section{Preliminaries}
In this section, we introduce some preliminaries for the main results of this paper.
Firstly, let's recall a variant of the Radial Lemma developed by Strauss
\cite{Strauss}.

\begin{lemma}\label{radial}
Suppose that $(V_0)$ holds, for every $u\in X_{r}$, there exist constants $R_0>0$
and $C>0$ independent of $u$ such that
\[
|u(x)|\leq C|x|^{-\frac{a+2}{4}}\|u\|,~\forall |x|\geq R_0.
\]
\end{lemma}

Proceeding as the proof of \cite[Theorem 2]{Su},
we have the following two imbedding results.

\begin{lemma}\label{imbedding1}
 For all $R>0$, the space $X$ can be continuously imbedded into $H^1(B_R(0))\triangleq
 \{u\in H^1(\R^2)|u\equiv u|_{B_R(0)}\}$ and $X\hookrightarrow L^\nu(B_R(0))$ is continuous
for all $\nu\geq1$.
\end{lemma}

\begin{lemma}\label{imbedding2}
Suppose that $(V,K)\in \mathcal{K}$ hold. Then, for every $2\leq s<+\infty$,
the imbedding $X_r\hookrightarrow L^s_K(\R^2)$ is compact.
\end{lemma}

Then, we consider the nonlocal Chern-Simons term \eqref{CSt}.
For any $u\in X_r$, denoting
\begin{equation}\label{gaugepart}
c(u)\triangleq\int_{\R^2}\frac{u^2}{|x|^2}\bigg(\int_{0}^{|x|}\frac{r}{2}u^2(r)dr\bigg)^2dx,
\end{equation}
and we can prove that
\begin{lemma}\label{imbedding3}
Suppose that $(V_0)$ holds, then for every $u\in X_r$,
 there exists a constant $C>0$ independent of $u$ such that
\[
0\leq c(u)\leq C\|u\|^6.
\]
\end{lemma}

\begin{proof}
Obviously, $c(u)\geq0$, $\forall u\in X_r$.
For all $u\in X_r$, by Lemma \ref{radial}, we derive
\begin{equation}\label{imbedding3a}
\int_0^{|x|}\frac{r}{2}u^2(r)dr\leq C\bigg(\int_0^{|x|}r^{-\frac{a}{2}}dr\bigg)\|u\|^2
=  C|x|^{\frac{2-a}{2}}\|u\|^2,
\end{equation}
where we have used $a<2$ in $(V_0)$. By means of the polar coordinate formula
and H\"{o}lder's inequality, we can conclude that
\begin{equation}\label{imbedding3b}
\int_0^{|x|}\frac{r}{2}u^2(r)dr=\frac{1}{4\pi}\int_{B_{|x|}(0)}u^2dy
\leq \frac{|x|}{4\sqrt{\pi}}\bigg(\int_{B_{|x|}(0)}u^4dy\bigg)^{\frac{1}{2}}.
\end{equation}
 Combining Lemmas \ref{imbedding1}-\ref{imbedding2}
and \eqref{imbedding3a}-\eqref{imbedding3b}, for every $u\in X_r$, we have that
\begin{align*}
c(u)& =\int_{B_{R_0}(0)}\frac{u^2}{|x|^2}\bigg(\int_{0}^{|x|}\frac{r}{2}u^2(r)dr\bigg)^2dx
+\int_{\R^2\backslash B_{R_0}(0)}\frac{u^2}{|x|^2}\bigg(\int_{0}^{|x|}\frac{r}{2}u^2(r)dr\bigg)^2dx \\
  & \leq
\frac{1}{16\pi} \int_{ B_{R_0}(0)} u^2 \bigg(\int_{B_{|x|}(0)}u^4dy\bigg) dx
     +C\|u\|^4 \int_{\R^2\backslash B_{R_0}(0)}\frac{u^2}{|x|^a}dx \\
      & \leq
\frac{1}{16\pi} \int_{ B_{R_0}(0)} u^2  dx\int_{B_{R_0}(0)}u^4dx
     +C\|u\|^6 \int_{ R_0}^{+\infty}r^{-\frac{3}{2}a}dr\\
     &\leq C\|u\|^6
\end{align*}
showing the desired result, where we have used $a>2/3$ in $(V_0)$.
The proof is complete.
\end{proof}

Next, inspired by the  Trudinger-Moser inequality in $\R^2$, see e.g.
 \cite{Cao,Bezerra1,Ruf,Lam,Yang,Alves1},
 we can follow the methods exploited in \cite{Albuquerque1,Albuquerque2}
  to establish the following two lemmas.

\begin{lemma}\label{nonlinearity1}
 Suppose that $(V,K)\in \mathcal{K}$ hold. Then, for every $u\in X_r$ and $\alpha>0$, we deduce that
 $K(|x|)(e^{\alpha u^2}-1)\in L^1(\R^2)$. Furthermore, if $0<\alpha<4\pi(1+\frac{b_0}{2})$, then
 \[
 \sup_{u\in X_r:\|u\|\leq1}\int_{\R^2}K(|x|)(e^{\alpha u^2}-1)dx<+\infty.
 \]
\end{lemma}

\begin{lemma}\label{nonlinearity2}
 Suppose that $(V,K)\in \mathcal{K}$ hold. Then, for every $\alpha>0$, if $u\in X_r$
 satisfies $\|u\|\leq \Upsilon <(\frac{4\pi}{\alpha}(1+\frac{b_0}{2}))^{1/2}$,
 there exists a constant $C=C(\Upsilon,\alpha)>0$ which is independent of $u$
 such that
 \[
 \int_{\R^2}K(|x|)(e^{\alpha u^2}-1)dx\leq C.
 \]
\end{lemma}

To study equation \eqref{mainequation1}  variationally, we  notice that $(f_1)$ and \eqref{definition} imply that for every $\epsilon>0$, there exists a
constant $C_\epsilon>0$ such that for all $\alpha>\alpha_0$,
\[
\max\{|f(t)t|, |F(t)|\}\leq \epsilon|t|^2+C_\epsilon|t|^s(e^{\alpha |t|^2}-1),~\forall t\in\R,
\]
where $s\in(2,+\infty)$. Thereby, for any $u\in X_r$, one has
\[
\int_{\R^2}K(|x|)\max\big\{|f(u)u|, |F(u)|\big\}dx
\leq \epsilon\int_{\R^2}K(|x|)|u|^2dx+C_\epsilon
\int_{\R^2} K(|x|)|u|^s(e^{\alpha |u|^2}-1)dx.
\]
With Lemmas \ref{imbedding2} and \ref{nonlinearity1} in hand, one can
determine the above integrals are well-defined. Moreover, we have that, for all $u\in X_r$,
there holds
\begin{equation}\label{fF}
	\int_{\R^2}K(|x|)\max\big\{| f(u)u|, | F(u)|\big\}dx
	\leq \epsilon\|u\|^2 +C_\epsilon \|u\|^s, \|u\|\leq  \Upsilon<\bigg(\frac{4\pi}{\alpha
		\overline{r}_2}(1+\frac{b_0}{2})\bigg)^{\frac{1}{2}}.
\end{equation}
In fact, let
  $r_1 > 1$ and $r_2> 1$ be such that $1/r_1 + 1/r_2 = 1$.
Thanks to \cite[Lemma 2.2]{Bezerra2},
let $\alpha>0$ and $r_2>1$, then for each $\overline{r}_2>r_2$, there is a constant
  $C=C(\overline{r}_2)>0$
such that $(e^{\alpha t^2}-1)^{r_2}\leq C(e^{\alpha \overline{r}_2 t^2}-1)$ for all $t\in\R$.
Combing the H\"{o}lder's inequality, Lemmas \ref{imbedding2} and \ref{nonlinearity2}, one has
\begin{align*}
 \int_{\R^2} K(|x|)|u|^s(e^{\alpha |u|^2}-1)dx&\leq \bigg( \int_{\R^2} K(|x|)|u|^{sr_1}dx\bigg)^{\frac{1}{r_1}}
\bigg( \int_{\R^2} K(|x|)(e^{\alpha |u|^2}-1)^{r_2}dx\bigg)^{\frac{1}{r_2}} \\
    &\leq C\|u\|^s \bigg(\int_{\R^2} K(|x|)(e^{\alpha \overline{r}_2\|u\|^2 (|u|^2/\|u\|^2)}-1) dx\bigg)^{\frac{1}{r_2}}\\
    &\leq C\|u\|^s, ~\text{if}~u\in X_r~\text{and}~\|u\|\leq  \Upsilon<\bigg(\frac{4\pi}{\alpha
    \overline{r}_2}(1+\frac{b_0}{2})\bigg)^{\frac{1}{2}},
\end{align*}
where we have used $s>2$ in the last second inequality. Finally, in view of $(g)$, one can conclude that there exists a constant $\overline{R}_0>R_0$ such that
$\overline{M}\triangleq \sup_{r\geq \overline{R}_0}g(r)/[r^\sigma V^{q/2}(r)]\in(0,+\infty)$,
where $R_0>0$ is given by Lemma \ref{radial}.
Therefore, for all $u\in X_r$, using Lemma \ref{imbedding1}
and the H\"{o}lder's inequality, one has
 \begin{align}\label{g}
\nonumber \int_{\R^2}g(|x|)|u|^qdx &\leq
\int_{B_{\overline{R}_0}(0)}g(|x|)|u|^qdx+\int_{\R^2\backslash B_{\overline{R}_0}(0)}g(|x|)|u|^qdx \\
\nonumber  & \leq  \|g\|_{L^\infty(B_{\overline{R}_0}(0))}\int_{B_{\overline{R}_0}(0)} |u|^qdx
  +\overline{M}\int_{\R^2\backslash B_{\overline{R}_0}(0)}|x|^\sigma V^{q/2}(|x|) |u|^qdx\\
 \nonumber &\leq C\|u\|^q+\overline{M}\bigg(\int_{\R^2\backslash B_{\overline{R}_0}(0)}|x|^{\frac{2\sigma}{2-q}} dx\bigg)^{\frac{2-q}{2}}
  \bigg(\int_{\R^2 } V (|x|) |u|^2dx\bigg)^{\frac{q}{2}}\\
  &\leq C\|u\|^q,
\end{align}
 where we have used $\sigma<q-2$ in the last inequality.

Summarize all of the above discussions,
we  know  that a function $u\in X_r$ is a (weak-radial) solution of
equation \eqref{mainequation1}, if $u$ is a critical point of the variational functional $J:X_r\to\R$ defined by
\[
J(u)=\frac{1}{2}\|u\|^2+\frac{\lambda}{2}
\int_{\R^2}\frac{u^2}{|x|^2}\bigg(\int_{0}^{|x|}\frac{r}{2}u^2(r)dr\bigg)^2dx- \int_{\R^2}K(|x|)F(u)dx
-\frac{\mu}{q}\int_{\R^2}g(|x|)|u|^qdx.
\]

\section{Proofs of Theorems \ref{maintheorem1} and \ref{maintheorem2}}

In this section, we shall present the proofs of Theorems \ref{maintheorem1} and
\ref{maintheorem2} in detail. To look for the mountain-pass type solutions for
equation \eqref{mainequation1}, firstly, we have to establish the validity of the mountain-pass geometry for $J$.

\begin{lemma}\label{geometry}
Suppose that $(V,K)\in \mathcal{K}$ and $(g)$ hold. Let $f$ satisfies \eqref{definition}, $(f_1)$,
and $(f_2)$, or $(f_5)$, then
there exists a constant $\mu_1>0$ such that if $\mu\in(0,\mu_1)$, the functional $J$ satisfies\\
\emph{(i)} there exist constants $\rho,\eta>0$ such that $J(u)\geq\eta$ for all $u\in X_r$ with $\|u\|=\rho$;\\
\emph{(ii)} there exists a function $e\in X_r$ with $\|e\|>\rho$ such that $J(e)\leq 0$.
\end{lemma}

\begin{proof}
(i) Let $\epsilon=1/4$ in \eqref{fF}, since $c(u)\geq 0$ for all $u\in X_r$, by \eqref{fF} and
\eqref{g}, we have
\begin{align*}
 J(u)&\geq \frac{1}{4}\|u\|^q\big(\|u\|^{2-q}-C_1\|u\|^{s-q}-C_2\mu\big),~
 \|u\|\leq  \Upsilon<\bigg(\frac{4\pi}{\alpha
    \overline{r}_2}(1+\frac{b_0}{2})\bigg)^{\frac{1}{2}}.
\end{align*}
Since $1<q<2<s$, there exist constants $\varsigma>0$ and $C_3>0$
such that $\varsigma^{2-q}-C_1\varsigma^{s-q}>C_3$.
Choosing $\mu_1=C_3/2C_2$, $\rho=\min\{\varsigma,\Upsilon\}$ and
$\eta=\varsigma^{q}/8C_3$, we can get the Point-(i).

 (ii) If $(f_5)$ holds, then we know that $\liminf_{t\to+\infty}F(t)/t^6=+\infty$ and the conclusion follows directly. If  $(f_2)$ holds, for $u_0\in X_r$ with $\|u_0\|\equiv 1$,
 then for every $C_F>0$, there is a sufficiently large $t_0>0$ such that $F(t_0u_0)\geq C_F(t_0 u_0)^6$.
Combining $(g)$ and Lemmas \ref{imbedding2}-\ref{imbedding3},
 one has
 \[
 J(t_0u_0)\leq \frac{t_0^2}{2}+Ct_0^6-CC_Ft_0^6<0
 \]
 if we increase the constant $C_F>0$ large enough. Let $e=t_0u_0$ with $t_0>\rho$, we then derive
 the Point-(ii).
 The proof is complete.
\end{proof}

By Lemma \ref{geometry} and the mountain-pass theorem in \cite{Willem}, a (PS) sequence of the functional $J:X_r\to \R$ at the level
\begin{equation}\label{Mountainpass1}
 c\triangleq\inf_{\gamma\in \Gamma}\max_{t\in [0,1]}J(\gamma(t))>0
\end{equation}
can be constructed, where the set of paths is defined as
\[
  \Gamma\triangleq\big\{\gamma\in C\big([0,1],X_r \big):\gamma(0)=0,
  ~ J(\gamma(1))<0\big\}.
\]
In other words, there exists a sequence $\{u_n\}\subset X_r $ such that
\begin{equation}\label{Mountainpass2}
  J(u_n)\to c,~ J^{\prime}(u_n)\to 0~ \text{as}~ n\to \infty.
\end{equation}
Next, since the nonlinearity is of  critical exponential
growth \eqref{definition}, we  need to prove that

\begin{lemma}\label{estimate}
Suppose that $(V,K)\in \mathcal{K}$, $f$ satisfies \eqref{definition} and
 $(f_1)-(f_3)$, and $(g)$ hold, then $c<\frac{\pi}{3\alpha_0}(1+\frac{b_0}{2})$.
\end{lemma}

\begin{proof}
Let $\varphi_0\in X_r$ be a cut-off function
satisfying $\varphi_0\in C_{0,r}^\infty(\R^2)$
defined by $0\leq \varphi_0(x)\leq1$ for every $x\in\R^2$, $\varphi_0(x)\equiv 1$
if $|x|\leq1/2$, $\varphi_0(x)\equiv 0$
if $|x|\geq1$ and $|\nabla \varphi_0|\leq1$
for all $x\in\R^2$. In view of the proof of
 Lemma \ref{imbedding3}, we can compute that
\begin{equation}\label{estimate1}
 c(\varphi_0)\leq \frac{1}{16\pi}\int_{B_1(0)}|\varphi_0|^2dx
 \int_{B_1(0)}|\varphi_0|^4dx\leq \frac{\pi}{16}.
\end{equation}
Combing $(f_3)$, $g(x)\geq 0$ in $(g)$ and \eqref{estimate1}, we infer that
\begin{align}\label{estimate1a}
\nonumber  J(\varphi_0) &\leq  \frac{1}{2}\int_{B_1(0)}[|\nabla \varphi_0|^2+V(|x|)|\varphi_0|^2]dx
+\frac{\lambda}{2}c(\varphi_0)-\int_{B_1(0)}K(|x|)F(\varphi_0)dx \\
\nonumber     &< \frac{(16+\lambda)\pi+16\|V\|_{L^1(B_1(0))}}{32}-\kappa\int_{B_{1/2}(0)}K(|x|) dx\\
    &\leq \frac{(16+\lambda)\pi+16\|V\|_{L^1(B_1(0))}}{32}-\kappa_1 \|K\|_{L^1(B_{1/2}(0))}
    =0.
\end{align}
 In particular, one can deduce from \eqref{estimate1a} that
 \begin{equation}\label{estimate2}
\int_{B_1(0)}[|\nabla \varphi_0|^2+V(|x|)|\varphi_0|^2]dx
+\lambda c(\varphi_0)<2\kappa_1 \|K\|_{L^1(B_{1/2}(0))}.
 \end{equation}
Choosing $\gamma_0(t)=t\varphi_0$, one easily knows that $\gamma_0(t)\in\Gamma$ by \eqref{estimate1a}.
According to the definition of $c$, then by \eqref{estimate2}
and $g(x)\geq 0$ in $(g)$, we have
\begin{align*}
  c &\leq \max_{t\in[0,1]}J(t\varphi_0)
  \leq \max_{t\in[0,1]}\bigg\{
  \frac{t^2}{2}\int_{B_1(0)}[|\nabla \varphi_0|^2+V(|x|)|\varphi_0|^2]dx
+ \frac{\lambda t^6}{2}c(\varphi_0)-\kappa t^p \|K\|_{L^1(B_{1/2}(0))}
  \bigg\}  \\
  &< \|K\|_{L^1(B_{1/2}(0))}\max_{t\in[0,1]}\bigg\{ \kappa_1 t^2
  -\kappa t^p\bigg\}\leq \|K\|_{L^1(B_{1/2}(0))}\max_{t\geq0}\big\{ \kappa_1 t^2
  -\kappa t^p\big\}\\
  &=  \frac{(p-2)\|K\|_{L^1(B_{1/2}(0))}}
  {p}\kappa_1\bigg(\frac{2\kappa_1}{p\kappa}\bigg)^{\frac{2}{p-2}}
  \leq\frac{\pi}{3\alpha_0}(1+\frac{b_0}{2})~\text{if}~\kappa\geq\kappa^*,
\end{align*}
where we have used the definition of $\kappa^*$. The proof is complete
\end{proof}

As a byproduct of Lemma \ref{estimate},
we can derive the following lemma which plays a significant role in recovering the compactness
caused by the critical exponential growth \eqref{definition}.

\begin{lemma}\label{bounded}
Suppose that $(V,K)\in \mathcal{K}$, $f$ satisfies \eqref{definition} and
 $(f_1)-(f_3)$, and $(g)$ hold.
Let $\{u_n\}\subset X_r$ be a (PS) sequence at the level $c$ of $J$, then there is
a constant $\mu_2>0$ such that
for any $\mu\in(0,\mu_2)$, there holds
\[
\limsup_{n\to\infty}\|u_n\|^2<\frac{4\pi}{\alpha_0}(1+\frac{b_0}{2}).
\]
\end{lemma}

\begin{proof}
By $(f_2)$ and \eqref{g}, it follows from the Young's inequality with $\epsilon=q/(6-q)$ that
\begin{align}\label{bounded0}
\nonumber  c +o(1)\|u_n\|  &\geq J(u_n)-\frac{1}{6} \langle J^\prime(u_n),u_n\rangle\\
\nonumber   &\geq  \frac{1}{3}\|u_n\|^2- \frac{\mu C(6-q)}{6q} \|u_n\|^q\\
\nonumber   &\geq \frac{1}{3}\|u_n\|^2-\frac{(6-q)}{6q} \bigg[
\epsilon \|u_n\|^2+\bigg(\frac{q}{2\epsilon}\bigg)^{q/(2-q)}\frac{2-q}{2}(\mu C)^{2/(2-q)}
\bigg]\\
&=\frac{1}{6}\|u_n\|^2-\frac{(6-q)(2-q)}{12q}
\bigg(\frac{6-q}{2}\bigg)^{q/(2-q)} (\mu C)^{2/(2-q)} ,
\end{align}
yielding that $\{u_n\}$ is a bounded sequence in $X_r$.
Since $\{u_n\}$ is a (PS) sequence at the level $c$,
 by \eqref{Mountainpass2} and \eqref{bounded0}, we obtain
 \begin{align}\label{bounded1}
\nonumber\limsup_{n\to\infty}\|u_n\|^2 & \leq
\limsup_{n\to\infty}\big[6J(u_n)- \langle J^\prime(u_n),u_n\rangle\big] +\frac{(6-q)(2-q)}{2q}
\bigg(\frac{6-q}{2}\bigg)^{\frac{q}{2-q} } (\mu C)^{\frac{2}{2-q} } \\
  &= 6c+\frac{(6-q)(2-q)}{2q}
\bigg(\frac{6-q}{2}\bigg)^{q/(2-q)}(\mu C)^{\frac{2}{2-q} }.
 \end{align}
 Because the constant $C>0$ comes from Lemma \ref{imbedding3},
 we can define
\begin{equation}\label{bounded2}
\mu_2\triangleq  C^{-1}\bigg[\frac{4\pi q}{(6-q)(2-q)\alpha_0}(1+\frac{b_0}{2})\bigg]^{\frac{2-q}{2}}
\bigg(
\frac{2}{6-q}\bigg)^{\frac{q}{2}}>0.
\end{equation}
Combing Lemma \ref{estimate} and \eqref{bounded1}-\eqref{bounded2}, we'll get the desired result.
The proof is complete.
\end{proof}

By Lemma \ref{bounded}, we have the following convergence properities.

\begin{lemma}\label{weak}
Suppose that $(V,K)\in \mathcal{K}$, $f$ satisfies \eqref{definition} and
 $(f_1)-(f_3)$, and $(g)$ hold.
Let $\{u_n\}\subset X_r$ be a (PS) sequence at the level $c$ of $J$, then, going to a subsequence if necessary,
for all $0<\mu< \mu_2$,
 there exists
a function $u\in X_r$ such that
\[
\left\{
  \begin{array}{ll}
   \displaystyle \lim_{n\to\infty}\int_{\R^2}K(|x|) f(u_n)(u_n-u)dx=0~\text{\emph{and}}~
 \lim_{n\to\infty}\int_{\R^2}K(|x|) f(u)(u_n-u)dx=0,\vspace{2mm} \\
  \displaystyle   \lim_{n\to\infty}c^\prime(u_n)[u_n-u]=0~\text{\emph{and}}~\lim_{n\to\infty}c^\prime(u)[u_n-u]=0 \\
   \displaystyle \lim_{n\to\infty}\int_{\R^2} g(|x|)|u_n|^{q-2}u_n(u_n-u)dx=0~\text{\emph{and}}~
\lim_{n\to\infty}\int_{\R^2} g(|x|)|u|^{q-2}u(u_n-u)dx=0.
  \end{array}
\right.
\]
\end{lemma}

\begin{proof}
Recalling Lemmas \ref{imbedding2} and \ref{bounded}, there exist a subsequence of $\{u_n\}$, still denoted by itself, and
a function $u\in X_r$ such that
\[
u_n\rightharpoonup u~ \text{in}~ X_r,~
u_n\to u~ \text{in}~ L^{ s }_K(\R^2)~
\text{for}~ s\in[2,+\infty)~\text{and}~
u_n\to u~ \text{a.e. in}~ \R^2.
\]
Choosing $r_1,r_2>1$ such that $\frac{1}{r_1} + \frac{1}{r_2}= 1$ as \eqref{fF}, then if $r_2$ is
sufficiently close to 1, there exists a constant $\overline{r}_2>r_2$
such that $\sup_{n\in \mathbb{N}}{\alpha \overline{r}_2\|u_n\|^2 } < 4\pi(1+\frac{b_0}{2})$
for all $\alpha>\alpha_0$ by Lemma \ref{bounded}. It follows from Lemma \ref{nonlinearity2} that
\begin{equation}\label{weak1}
 \int_{\R^2} K(|x|)(e^{\alpha |u_n|^2}-1)^{r_2}dx \leq
 C\int_{\R^2} K(|x|)(e^{\alpha \overline{r}_2\|u_n\|^2 (|u_n|^2/\|u_n\|^2)}-1) dx\leq C<+\infty.
\end{equation}
Notice that $r_2$ is very
close to 1, without loss of generality, then we can suppose that $r_1>2$.
Let $r_3,r_4>1$ such that $\frac{1}{r_3} + \frac{1}{r_4}= 1$, since $\{u_n\}$ is bounded and $u_n\rightharpoonup u$,
by Lemma \ref{imbedding2}, we have that
\begin{equation}\label{weak2}
\left\{
  \begin{array}{ll}
  |u_n|_{K,2}~\text{and}~|u_n|_{K,r_1r_3(s-1)}~\text{are uniformly bounded with respect to}~n\in\mathbb{N}, \\
 |u_n-u|_{K,2}\to0~\text{and}~|u_n-u|_{K,r_1r_4}\to0.
  \end{array}
\right.
\end{equation}
Combing \eqref{weak1} and \eqref{weak2}, we obtain
\begin{eqnarray*}
  &&\bigg|\int_{\R^2}K(|x|) f(u_n)(u_n-u)dx\bigg| \\
  &\leq& C\bigg(\int_{\R^2}K(|x|)|u_n|^{r_1(s-1)}|u_n-u|^{r_1}dx\bigg)^{\frac{1}{r_1}}
 \bigg(\int_{\R^2} K(|x|)(e^{\alpha |u_n|^2}-1)^{r_2}dx\bigg)^{\frac{1}{r_2}}\\
 &&+C|u_n|_{K,2}|u_n-u|_{K,2}\\
 &\leq& C\bigg(\int_{\R^2}K(|x|)|u_n|^{r_1(s-1)}|u_n-u|^{r_1}dx\bigg)^{\frac{1}{r_1}}+o_n(1)\\
  &\leq&C\bigg(\int_{\R^2}K(|x|)|u_n|^{r_1r_3(s-1)} dx\bigg)^{\frac{1}{r_1r_3}}
  \bigg(\int_{\R^2}K(|x|) |u_n-u|^{r_1r_4}dx\bigg)^{\frac{1}{r_1r_4}}+o_n(1)=o_n(1).
\end{eqnarray*}
Similarly, we can deduce that $\int_{\R^2}K(|x|) f(u)(u_n-u)dx\to0$ as $n\to\infty$.

Let's define $v_n\triangleq u_n-u\rightharpoonup 0$
as $n\to\infty$, then arguing as \eqref{solution}, one has,
\begin{align*}
c^\prime(u_n)[v_n]&=\int_{\R^2}\frac{u_n^2}{|x|^2}
  \bigg(\int_{0}^{|x|}\frac{r}{2}u_n^2(r)dr\bigg)\bigg(\int_{0}^{|x|}ru_n(r)v_n(r)dr\bigg)dx  +\int_{\R^2}\frac{u_nv_n}{|x|^2}\bigg(\int_{0}^{|x|}\frac{r}{2}u_n^2(r)dr\bigg)^2dx\\
  & \triangleq c_n^1+c_n^2,
\end{align*}
where
 \begin{align*}
  |c_n^1| & = \bigg|\int_{\R^2}\frac{u_n^2}{|x|^2}
  \bigg(\int_{0}^{|x|}\frac{r}{2}u_n^2(r)dr\bigg)\bigg(\int_{0}^{|x|}ru_n(r)v_n(r)dr\bigg)dx \bigg|\\
     &  \leq \bigg(\int_{\R^2}\frac{u_n^2}{|x|^2}\bigg(\int_{0}^{|x|}\frac{r}{2}u_n^2(r)dr\bigg)^2dx\bigg)^{\frac{1}{2}}
 \bigg(\int_{\R^2}\frac{u_n^2}{|x|^2}\bigg(\int_{0}^{|x|}ru_n(r)v_n(r)dr\bigg)^2dx\bigg)^{\frac{1}{2}}\\
 &\leq C\|u_n\|^3 \bigg(\int_{\R^2}\frac{u_n^2}{|x|^2}\bigg(\int_{0}^{|x|}ru_n^2(r) dr\bigg)
 \bigg(\int_{0}^{|x|}r v_n^2(r)dr\bigg) dx\bigg)^{\frac{1}{2}}\\
 &\leq C\|u_n\|^3 \bigg(\int_{\R^2}\frac{u_n^2}{|x|^2}\bigg(\int_{0}^{|x|}ru_n^2(r) dr\bigg)^2
  dx\bigg)^{\frac{1}{4}}
 \bigg(\int_{\R^2}\frac{u_n^2}{|x|^2}
 \bigg(\int_{0}^{|x|}r v_n^2(r)dr\bigg)^2 dx\bigg)^{\frac{1}{4}}\\
&\leq C\|u_n\|^{\frac{9}{2}}\bigg(\int_{\R^2}\frac{u_n^2}{|x|^2}
 \bigg(\int_{0}^{|x|}r v_n^2(r)dr\bigg)^2 dx\bigg)^{\frac{1}{4}}
 \end{align*}
 and
\[
  |c_n^2|    =\bigg|\int_{\R^2}\frac{u_nv_n}{|x|^2}\bigg(\int_{0}^{|x|}\frac{r}{2}u_n^2(r)dr\bigg)^2dx \bigg|
  \leq C\|u_n\|^3 \bigg(\int_{\R^2}\frac{ v_n^2}{|x|^2}\bigg(\int_{0}^{|x|}\frac{r}{2}u_n^2(r)dr\bigg)^2dx\bigg)^{\frac{1}{2}}.
\]
Obviously, to verify $c^\prime(u_n)[v_n]\to 0$, it suffices to show that
\begin{equation}\label{weak3a}
\int_{\R^2}\frac{u_n^2}{|x|^2}
 \bigg(\int_{0}^{|x|}r v_n^2(r)dr\bigg)^2 dx\to0~\text{and}~
 \int_{\R^2}\frac{ v_n^2}{|x|^2}\bigg(\int_{0}^{|x|}\frac{r}{2}u_n^2(r)dr\bigg)^2dx\to0.
\end{equation}
For every $\epsilon>0$, there is a constant $R_\epsilon=\epsilon^{\frac{2}{2-3a}}>0$ such that
$R_\epsilon>R_0$ if $\epsilon>0$ is sufficiently small since $a>2/3$ in $(V_0)$, where $R_0>0$ is given by Lemma \ref{radial}.
It's similar to \eqref{imbedding3a} that
\begin{equation}\label{weak3}
\int_{\R^2\backslash B_{R_\epsilon}(0)}\frac{u_n^2}{|x|^2}
 \bigg(\int_{0}^{|x|}r v_n^2(r)dr\bigg)^2 dx
 \leq C\|u_n\|^2\|v_n\|^4  \int_{ R_\epsilon}^{+\infty}r^{-\frac{3}{2}a}dr
 \leq CR_\epsilon^{1-\frac{3a}{2}} =C\epsilon.
\end{equation}
In view of Lemma \ref{imbedding1} and \eqref{imbedding3b}, we can derive
\begin{equation}\label{weak4}
 \int_{B_{R_\epsilon}(0)}\frac{u_n^2}{|x|^2}
 \bigg(\int_{0}^{|x|}r v_n^2(r)dr\bigg)^2 dx
 \leq \frac{C\|u_n\|^2}{16\pi}\int_{B_{R_\epsilon}(0)}v_n^4dx.
\end{equation}
Fixed a $\epsilon>0$ in \eqref{weak3} and \eqref{weak4}, then
letting $n\to\infty$ in \eqref{weak3} and \eqref{weak4}, subsequently, $\epsilon\to0^+$ in
\eqref{weak3}, we get the first part of \eqref{weak3a}.
Analogously, one can accomplish the proof of the second part of \eqref{weak3a}.

 Recalling \eqref{g}, one can repeat the methods utilized in \eqref{weak3} and \eqref{weak4}
 to conclude that
$\int_{\R^2} g(|x|)|u_n|^{q-2}u_n(u_n-u)dx\to0$ and $\int_{\R^2} g(|x|)|u|^{q-2}u(u_n-u)dx\to0$.
So, we'll omit the details.
\end{proof}

Now, combining Lemmas \ref{geometry}-\ref{weak},
we can find a solution of equation\eqref{mainequation1} with positive energy.

\begin{lemma}\label{firstsolution}
Suppose that $(V,K)\in \mathcal{K}$, $f$ satisfies \eqref{definition} and
 $(f_1)-(f_3)$, and $(g)$ hold.
Then there exists a constant $\mu_0>0$ such that equation \eqref{mainequation1} admits at
  least a nontrivial solution with positive energy for any $\lambda>0$ and $\mu\in(0,\mu_0)$.
\end{lemma}
\begin{proof}
 Recalling the definition of $\mu_1$ introduced in Lemma \ref{geometry}-(i) and \eqref{bounded2},
set $\mu_0\triangleq\min\{\mu_1,\mu_2\}>0$,
 then all of the conclusions in Lemmas \ref{geometry} and \ref{bounded} remain true.
So, there exists a sequence $\{u_n\}\subset X_r$ verifying \eqref{Mountainpass2}.
Passing to a subsequence of $\{u_n\}$ if necessary, by Lemma \ref{bounded}, we derive
\[
\begin{gathered}
o(1)=\langle J^\prime(u_n)-J^\prime(u),u_n-u\rangle
 =\|u_n-u\|^2
+\lambda\big(c^\prime(u_n)[u_n-u]-c^\prime(u)[u_n-u]\big)\hfill\\
\ \ \ \ \ \ \ \ \ \ \ \ +\int_{\R^2}K(|x|)f(u_n)(u_n-u)dx
-\int_{\R^2}K(|x|)f(u)(u_n-u)dx\hfill\\
\ \ \ \ \ \ \ \ \ \ \ \ + \mu\int_{\R^2}g(|x|)|u_n|^{q-2}u_n(u_n-u) dx-\mu\int_{\R^2}g(|x|)|u|^{q-2}u(u_n-u) dx,\hfill\\
\end{gathered}
\]
which together with Lemma \ref{weak} indicates
 that $u_n\to u$ in $X_r$. Consequently,
 we proved that $J^\prime(u)=0$ and $J(u)=c>0$ by \eqref{Mountainpass1}. The proof is complete.
\end{proof}

To look for the second solution of
equation \eqref{mainequation1}, we need the following lemma.

\begin{lemma}\label{variational}
(Ekeland's variational principle \cite[Theorem 1.1]{Ekeland})
Let $E$ be a complete metric space and $H:E\to \R\cup\{+\infty\}$ be lower semicontinuous,
 bounded from below. Then for any $\epsilon>0$, there exists some point $v\in E$ with
$$
H(v)\leq \inf_EH+\epsilon,~ H(w)\geq H(v)-\epsilon d(v,w),~\forall  w\in E.
$$
\end{lemma}

Then, we can construct a (PS) sequence for the functional $J$ with negative energy.

\begin{lemma}\label{2Mountainpass}
Suppose that $(V,K)\in \mathcal{K}$, the nonlinearity $f$ satisfies \eqref{definition} and
 $(f_1)-(f_3)$, and $(g)$ hold.
Then there exists a (PS) sequence for the functional $J$ at the level below zero.
\end{lemma}

\begin{proof}
For $\rho>0$ given by Lemma \ref{geometry}-(i), we define
$$
\overline{B}_\rho=\{u\in X_r|\|u\|\leq \rho\},~ \partial B_\rho=\{u\in  X_r|\|u\|= \rho\}.
$$
Evidently, $\overline{B}_\rho$ is a complete metric space with the distance $d(u,v)\triangleq\|u-v\|$.
It is obvious that $J$ is lower semicontinuous and bounded from below on $\overline{B}_\rho$. We claim that
\begin{equation}\label{negativeminimum}
  \widetilde{c}\triangleq\inf\{J(u)|u\in \overline{B}_\rho\}<0.
\end{equation}
Indeed, choosing a nonnegative function $\psi\in C_{0,r}^{\infty}(\R^2)$, then by $(f_1)$, one has
\begin{equation*}
  \lim_{\theta\to 0}\frac{J(\theta\psi)}{\theta^q}=-\frac{\mu}{q}\int_{\R^2}g(x)|\psi|^q dx<0.
\end{equation*}
Thus, there is a sufficiently small $t_\psi>0$ such that $\|t_\psi\psi\|\leq \rho$ and $J(t_\psi\psi)<0$, which imply
that \eqref{negativeminimum} holds. By Lemma \ref{variational}, for any $n\in \mathbb{N}$, there exists
a function $\widetilde{u}_n$ such that
\[
\widetilde{ c} \leq J(\widetilde{u}_n)\leq \widetilde{c}  +\frac{1}{n}
~ \text{and}~
  J(v)\geq J(\widetilde{u}_n)-\frac{1}{n}\|\widetilde{u}_n-v\|,~ \forall v\in  \overline{B}_\rho.
\]
Then, a standard procedure indicates
 that the sequence
 $\{\widetilde{u}_n\}$ is a bounded $(PS)_{\widetilde{c}}$ sequence of $J$. The proof is complete.
\end{proof}

If $\{\widetilde{u}_n\}\subset X_r$ is a (PS) sequence of $J$ at the level $\widetilde{c}<0$,
similar to Lemma \ref{bounded}, we derive the following lemma.

\begin{lemma}\label{2bounded}
Suppose that $(V,K)\in \mathcal{K}$, $f$ satisfies \eqref{definition} and
 $(f_1)-(f_3)$, and $(g)$ hold.
Let $\{\widetilde{u}_n\}\subset X_r$ be a $(PS)_{\widetilde{c}}$ sequence of $J$ with $\widetilde{c}<0$,
then there is a constant $\mu_3>0$ such that for all $\mu\in(0,\mu_3)$, there holds
\[
\limsup_{n\to\infty}\|\widetilde{u}_n\|^2<\frac{4\pi}{\alpha_0}(1+\frac{b_0}{2}).
\]
\end{lemma}

\begin{proof}
By using $(f_2)$ and \eqref{g},
\[
\frac{1}{3}\|\widetilde{u}_n \|^2- \frac{\mu(6-q)}{6q}C\|\widetilde{u}_n\|^q
\leq J( \widetilde{u}_n )
-\frac{1}{6} \langle J^\prime( \widetilde{u}_n ),
 \widetilde{u}_n \rangle,
\]
which together with $\{\widetilde{u}_n\}\subset X_r$ is a (PS) sequence of $J$ at the level $\widetilde{c}<0$
gives that
\[
\limsup_{n\to\infty}\| \widetilde{u}_n \|^2
\leq \bigg(\frac{\mu(6-q)C}{2q}\bigg)^{\frac{1}{2-q}}.
\]
Hence, we can define
\begin{equation}\label{2bounded1}
\mu_3\triangleq  \frac{2q}{(6-q)C}\bigg(\frac{4\pi}{\alpha_0}(1+\frac{b_0}{2})\bigg)^{2-q}>0,
\end{equation}
then we obtain the desired result. The proof is complete.
\end{proof}

We establish the existence of solutions with negative energy for equation \eqref{mainequation1}.

\begin{lemma}\label{secondsolution}
Suppose that $(V,K)\in \mathcal{K}$, $f$ satisfies \eqref{definition} and
 $(f_1)-(f_3)$, and $(g)$ hold. Then there exists a constant $\mu^0>0$ such that equation \eqref{mainequation1}
admits at least a nontrivial solution with negative energy for any $\lambda>0$ and $\mu\in(0,\mu^0)$.
\end{lemma}
\begin{proof}
Set $\mu^0\triangleq\min\{\mu_1,\mu_3\}>0$, where $\mu_1>0$ comes from Lemma \ref{geometry}-(i) and $\mu_3>0$ is defined by \eqref{2bounded1},
respectively. It follows from Lemma \ref{2Mountainpass} and \eqref{negativeminimum}
 that there exists a sequence $\{\widetilde{u}_n\}\subset X_r$ such that
\[
J(\widetilde{u}_n)\to \widetilde{c}<0,~ J^{\prime}(\widetilde{u}_n)\to 0~  \text{as}~ n\to \infty.
\]
According to Lemma \ref{2bounded}, up to a subsequence if necessary,
 there exists a  function $\widetilde{u}\in X_r$ such that $\widetilde{u}_n\rightharpoonup \widetilde{u}$
 in $X_r$, $\widetilde{u}_n\to\widetilde{u}$ in $L^s_K(\R^2)$ with $s\in[2,+\infty)$
 and $\widetilde{u}_n\to \widetilde{u}$ a.e. in $\R^2$. By
using the similar arguments in Lemma \ref{firstsolution}, we can conclude
 that $\widetilde{u}_n\to \widetilde{u}$ in $X_r$. Therefore, we have
 obtained that $J^\prime(\widetilde{u})=0$ and $J(\widetilde{u})=\widetilde{c}<0$.
The proof is complete.
\end{proof}

Nowt, we are in a position to finish the proof of Theorem \ref{maintheorem1}.

\begin{proof}[\textbf{\emph{Proof of Theorem \ref{maintheorem1}}}]
Set $\mu_*\triangleq\min\{\mu_0,\mu^0\}>0$, it follows from Lemmas \ref{firstsolution} and \ref{secondsolution} that \eqref{mainequation1} has two nontrivial
 solutions $u$ and $\widetilde{u}$ for any $\lambda>0$ and $\mu\in(0,\mu_*)$. On the other hand, the fact $J(\widetilde{u})<0<J(u)$ shows
that $u$ and $\widetilde{u}$ are two different solutions of equation
\eqref{mainequation1}. The proof is complete.
\end{proof}

Next, let's focus on the proof of Theorem \ref{maintheorem2}. In consideration of
the process of the proof of Theorem \ref{maintheorem1}, the essential
difference is then how to enforce the same estimate on
the mountain-pass value \eqref{Mountainpass1} when $(f_3)$ is replaced with $(f_4)$ and $(f_5)$.
To this aim, for a fixed constant $r_0\in(0,1]$, we consider the Moser sequence defined by
\[
\overline{w}_n(x)\triangleq \frac{1}{\sqrt{2\pi}}
\left\{
  \begin{array}{ll}
    \sqrt{\log n}, & \text{if}~0\leq |x|\leq \frac{r_0}{n}, \vspace{2mm}\\
    \frac{\log(\frac{r_0}{|x|})}{\sqrt{\log n}}, & \text{if}~ \frac{r_0}{n}
    <|x|\leq r_0, \vspace{2mm}\\
    0, & \text{if}~|x|>r_0,
  \end{array}
\right.
\]
see \cite{Bezerra1,Bezerra3,AY,Dong} for example. Since $\supp \overline{w}_n\subset B_{r_0}(0)$,
it's simple to check that $\{\overline{w}_n\}\subset X_r$ if $(V_0)$ holds.
What's more, we can derive the following lemma.

\begin{lemma}\label{Moser}
Suppose that $(V_0)$ holds and let $M_1\triangleq\sup_{r\in(0,1]}V(r)/r^{a_0}\in(0,+\infty)$.
Then $\|\overline{w}_n\|^2\leq 1+\delta_n$ with
$\delta_n>0$ and $\delta_n\log n\to  2M_1 r_0^{a_0+2}(a_0+2)^{-3}$ as $n\to
\infty$.
In particular, $c(\overline{w}_n)$ defined in \eqref{gaugepart} goes to $0$ as $n\to
\infty$.
\end{lemma}

\begin{proof}
The proof is mainly inspired by
\cite{Bezerra1,Bezerra3,AY,Lam,Dong}, we sketch it here for the convenience of the interested reader.
Obviously, $|\nabla\log(r_0/|x|)|^2=1/|x|^2$, then
\begin{equation}\label{Moser0a}
\int_{\R^2}|\nabla\overline{w}_n|^2dx=\frac{1}{2\pi\log n}\int_{B_{r_0}(0)\backslash B_{r_0/n}(0)}
\frac{1}{|x|^2}dx=\frac{1}{ \log n}\int^{r_0}_{r_0/n} \frac{1}{r}dr=1.
\end{equation}
By means of the polar coordinate formula,
we can infer from $(V_0)$ that
\begin{align}\label{Moser0b}
\nonumber  \int_{\R^2}V(|x|)| \overline{w}_n|^2dx& =\int_{B_{r_0 }(0)}V(|x|)| \overline{w}_n|^2dx
\leq M_1\int_{B_{r_0 }(0)} |x|^{a_0} |\overline{w}_n|^2dx \\
\nonumber     &= M_1\int_{B_{r_0/n }(0)} |x|^{a_0} |\overline{w}_n|^2dx
   +M_1\int_{B_{r_0 }(0)\backslash B_{r_0/n }(0)} |x|^{a_0} |\overline{w}_n|^2dx\\
     &=\frac{M_1r_0^{a_0+2}\log n}{(a_0+2)n^{a_0+2}}+\frac{M_1}{\log n}\int^{r_0}_{r_0/n}
   \log^2\bigg(\frac{r_0}{r}\bigg) r^{a_0+1}dr \triangleq\delta_n.
\end{align}
As a direct consequence of \eqref{Moser0a} and \eqref{Moser0b}, one derives that
$\|\overline{w}_n\|^2\leq 1+\delta_n$ for all $n\in \mathbb{N}$, with $\delta_n>0$.
 Moreover, by some elementary computations, there holds
 \begin{align*}
  \delta_n & =\frac{M_1r_0^{a_0+2}\log n}{(a_0+2)n^{a_0+2}}+\frac{M_1}{\log n}\int^{r_0}_{r_0/n}
   \log^2\bigg(\frac{r_0}{r}\bigg) r^{a_0+1}dr \\
     & =\frac{M_1r_0^{a_0+2}\log n}{(a_0+2)n^{a_0+2}}+\frac{M_1r_0^{a_0+2}}{\log n}\int^{\log n}_{ 0 }
   t^2e^{-(a_0+2)t}dt\\
   &=\frac{M_1r_0^{a_0+2}\log n}{(a_0+2)n^{a_0+2}}-\frac{M_1r_0^{a_0+2}}{\log n}
   \bigg(\frac{t^2}{a_0+2}+\frac{2t}{(a_0+2)^2}+\frac{2}{(a_0+2)^3}\bigg)e^{-(a_0+2)t}\bigg|_0^{\log n}\\
   &=\frac{2M_1r_0^{a_0+2}}{(a_0+2)^3\log n}- \frac{2M_1r_0^{a_0+2} }{(a_0+2)^3n^{a_0+2}\log n}-
   \frac{2M_1r_0^{a_0+2}}{(a_0+2)^2n^{a_0+2} }
 \end{align*}
 indicating that $\delta_n\log n\to2M_1r_0^{a_0+2}/(a_0+2)^{3}$ as $n\to\infty$.

Next, we verify that $c(\overline{w}_n)\to0$ as $n\to\infty$.
Let's claim that $\overline{w}_n\to 0$ in $L^2(\R^2)$. Indeed,
since $\supp \overline{w}_n\subset B_{r_0}(0)$, one has
\[
\int_{\R^2}\overline{w}_n^2dx=
 \int_{B_{r_0}(0)\backslash B_{r_0/n}(0)}\overline{w}_n^2dx
+\int_{B_{r_0/n}(0)}\overline{w}_n^2dx\triangleq  \Sigma_n^1+\Sigma_n^2.
\]
By some elementary computations, we have
\begin{align}\label{Moser1}
\nonumber  \Sigma_n^1 &=\frac{1}{2\pi\log n}
\int_{B_{r_0}(0)\backslash B_{r_0/n}(0)} \log^2\bigg(\frac{r_0}{|x|}\bigg)dx
=\frac{1}{ \log n}
\int^{r_0}_{r_0/n} \log^2\bigg(\frac{r_0}{r}\bigg)rdr\\
 \nonumber   &=\frac{r_0^2}{ \log n}
\int_{0}^{\log n} t^2e^{-2t} dt=
-\frac{r_0^2}{4 \log n}\frac{2t^2+2t+1}{e^{2t}}\bigg|_{0}^{\log n}\\
&=\frac{r_0^2}{4 \log n}\bigg(1-\frac{2\log^2 n+2\log n+1}{n^2}\bigg) \to0
~\text{as}~n\to\infty
\end{align}
and
\begin{equation}\label{Moser2}
\Sigma_n^2 =\frac{\log n}{2\pi}\int_{B_{r_0/n}(0)} dx
=\frac{r_0^2\log n}{2n^2}\to0~\text{as}~n\to\infty.
\end{equation}
Combing \eqref{Moser1}-\eqref{Moser2}, we can conclude that
$\overline{w}_n\to 0$ in $L^2(\R^2)$ as $n\to\infty$.
Recalling that \eqref{imbedding3b}, we utilize Lemma \ref{imbedding1} to infer that
\[
c(\overline{w}_n)\leq \frac{1}{16\pi}\int_{B_{r_0}(0)}\overline{w}_n^2dx
\int_{B_{r_0}(0)}\overline{w}_n^4dx\leq \frac{C\|\overline{w}_n\|^4}{16\pi}\int_{\R^2}\overline{w}_n^2dx
\leq \frac{C(1+\delta_n)^2}{16\pi}\int_{\R^2}\overline{w}_n^2dx
\]
showing the desired result. The proof of this lemma is finished.
\end{proof}

\begin{lemma}\label{2estimate}
Suppose that $(V,K)\in \mathcal{K}$, $f$ satisfies \eqref{definition}, $(f_1)-(f_2)$ and
$(f_4)-(f_5)$, and (g) hold. If we also suppose that $\liminf_{r\to0^+}K(r)/r^{(b_0-22)/12}>0$,
then $c<\frac{\pi}{3\alpha_0}(1+\frac{b_0}{2})$.
\end{lemma}

\begin{proof}
Let's define $w_n=\overline{w}_n/\sqrt{1+\delta_n}$,
then $\|w_n\|\leq 1$ and $c(w_n)=c(\overline{w}_n)/(1+\delta_n)^3 \to0$ by Lemma \ref{Moser}.
By $(g)$, to end with
 the proof, it's enough to show that there exists some $n_0\in \mathbb{N}$ such that
\[
\max_{t\geq0}\bigg\{\frac{t^2}{2}+\frac{\lambda t^6}{2}c(w_{n_0})
-\int_{\R^2}K(|x|)F(tw_{n_0})dx\bigg\}<\frac{\pi}{3\alpha_0}(1+\frac{b_0}{2}).
\]
Indeed, we can chose a sufficiently large $t_0>0$ to satisfy $\|t_0w_{n_0}\|>\rho$
and $J(t_0w_{n_0})<0$ by Lemma \ref{geometry}, then $\gamma_0(t)=tt_0w_{n_0}\in\Gamma$. Since $g(x)\geq0$
in $(g)$, one has
\begin{align*}
  c &=\inf_{\gamma\in \Gamma}\max_{t\in[0,1]}J(\gamma(t))\leq \max_{t\in[0,1]}J(\gamma_0(t))
\leq \max_{t\geq0}J(tw_{n_0}) \\
    & \leq \max_{t\geq0}\bigg\{\frac{t^2}{2}+\frac{\lambda t^6}{2}c(w_{n_0})
-\int_{\R^2}K(|x|)F(tw_{n_0})dx\bigg\}.
\end{align*}
On the contrary, suppose that for all $n\in \mathbb{N}$, there is a constant $t_n>0$
 such that
 \begin{equation}\label{2estimate1}
 \frac{t_n^2}{2}+\frac{\lambda t_n^6}{2}c(w_{n})
-\int_{\R^2}K(|x|)F(t_nw_{n})dx\geq \frac{\pi}{3\alpha_0}(1+\frac{b_0}{2}).
 \end{equation}
Moveover, it's obvious to check that
 \begin{equation}\label{2estimate2}
 t_n^2+3\lambda t_n^6c(w_{n})=\int_{\R^2}K(|x|)f(t_nw_{n})t_nw_{n}dx.
 \end{equation}
From $(f_4)$ and $(f_5)$, for all $\epsilon\in(0,\beta_0)$, there exists a constant
$R_\epsilon=R(\epsilon)>0$ such that
\[
 f(t)t \geq M_0^{-1}(\beta_0-\epsilon)t^{\vartheta+1}e^{\alpha_0|t|^2},~\forall t
\geq R_\epsilon.
\]
By the additional assumption of $K(r)$, there exist constants $C>0$ and $\overline{r}_0$ such that
\[
K(|x|)\geq C|x|^{(b_0-22)/12},~\forall x\in B_{\overline{r}_0}(0).
\]
Thanks to \eqref{2estimate1}, $\{t_n\}$ is bounded below by some positive constant.
Since $B_{r_0/n}(0)\subset B_{\overline{r}_0}(0)$ for some sufficiently large $n\in \mathbb{N}$,
one deduces that $t_nw_n\geq R_\epsilon$ on $B_{r_0/n}(0)$. Then, on one hand, by \eqref{2estimate2},
we can obtain that
\begin{align}\label{2estimate3}
\nonumber  t_n^2 +3\lambda  t_n^6 c(w_{n})  &  \geq C M_0^{-1}(\beta_0-\epsilon)(t_nw_n)^{\vartheta+1}e^{\alpha_0|t_nw_n|^2}\int_{B_{r_0/n}(0)}
  |x|^{(b_0-22)/12}dx \\
   & \geq \frac{24\pi C(\beta_0-\epsilon)}{M_0 (b_0+2)}  t_n^{\vartheta+1}
   \bigg(\frac{\log n}{2\pi(1+\delta_n)}\bigg)^{\frac{\vartheta+1}{2}}
   \exp\bigg(\frac{\alpha_0t_n^2\log n}{2\pi(1+\delta_n)}\bigg)\bigg(\frac{r_0}{n}\bigg)^{\frac{b_0+2}{12}}
\end{align}
which, together with the fact that $c(w_n)\leq C\|w_n\|^6\leq C<+\infty$,
implies that $\{t_n\}$ is uniformly bounded in $n\in \mathbb{N}$.
Up to a subsequence if necessary, there is a constant $t_0\in[0,+\infty)$
 such that $t_n\to t_0$.

On the other hand, using $(f_1)$, $c(w_n)\to0$ and \eqref{2estimate1}, we have
\begin{equation}\label{2estimate4}
 t_0^2\geq \frac{2\pi}{3\alpha_0}(1+\frac{b_0}{2}).
\end{equation}
Taking  $\epsilon=\beta_0/2$ in \eqref{2estimate3} and applying \eqref{2estimate4}, we also obtain
\begin{align*}
  (1-\vartheta)\log t_0+o(1) & \geq C+C\log(\log n)+C\bigg(\frac{\alpha_0}{2\pi}t_0^2-\frac{b_0+2}{12}\bigg)
  \log n+ o(\log n)  \\
   &\geq C+C\log(\log n)+\frac{C}{12}(b_0+2)\log n+o( \log n ),
\end{align*}
 yields a contradiction since $b_0>-2$ in $(K_0)$. The proof is complete.
\end{proof}

With Lemma \ref{2estimate} in mind, we can finish the proof of Theorem \ref{maintheorem2}.

\begin{proof}[\textbf{\emph{Proof of Theorem \ref{maintheorem2}}}]
Proceeding as Lemmas \ref{firstsolution} and \ref{secondsolution},
where Lemma \ref{estimate} is replaced by Lemma \ref{2estimate}, we can conclude that
there exists a constant $\mu_{**}>0$ such that
equation  \eqref{mainequation1}
 has two different nontrivial solutions for any $\lambda>0$ and $\mu\in(0,\mu_{**})$. The proof is complete.
\end{proof}

\section{Proof of Theorem \ref{maintheorem3}}
In this section, we mainly discuss the existence of infinitely many solutions for equation \eqref{mainequation1}.
For this purpose, we shall exploit the new symmetric mountain-pass theorem
developed by Kajikiya \cite{Kajikiya}. For simplicity, we'll always suppose that all of the assumptions in
Theorem \ref{maintheorem3} are satisfied in this section.

Now,
let's recall some notations with respect to the Krasnoselskii's genus theory in \cite{Krasnoselskii} for the
sake of completeness and reader's convenience. Suppose $E$ to be a Banach space and we
 denote by $\Sigma$ the class of all closed subsets $A\subset E\backslash \{0\}$ that are symmetric
corresponding to the origin,
that is, $u\in A$ implies that $-u\in A$.

\begin{definition}
If $A\in \Sigma$, the Krasnoselskii's genus $\gamma(A)$ of $A$ is defined by the least positive
integer $n$ such that there is an odd mapping $\varphi\in C(A,\R^n)$ such that $\varphi(x)\neq0$ for any
$x\in A$. If $n$ does not exist, we set $\gamma(A)=\infty$. Furthermore, we set $\gamma(\emptyset)= 0$.
\end{definition}

In the following, we will bring in some necessary properties of the genus for the proof of Theorem \ref{maintheorem3}
and the complete introduction to it can be found in e.g. \cite{Krasnoselskii,Rabinowitz,Kajikiya}.

\begin{proposition}\label{gamma}
Let $A$ and $B$ be closed symmetric subsets of $E$ which do not contain the origin. Then
the following properties are true:
\begin{itemize}
  \item [(1)] If there exists an odd continuous mapping from $A$ to $B$, then $\gamma(A)\leq \gamma(B)$;
\item [(2)] If there is an odd homeomorphism from $A$ to $B$, then $\gamma(A)=\gamma(B)$;
\item [(3)] If $A\subset B$, then $\gamma(A)\leq\gamma(B)$;
\item [(4)] $\gamma(A\cup B)\leq \gamma(A)+\gamma(B)$;
\item [(5)] If $\gamma(B)<\infty$, then $\gamma(\overline{A\backslash B})\geq \gamma(A)-\gamma(B)$;
\item [(6)] If $A$ is compact, then $\gamma(A)<\infty$ and there exists a constant $\delta> 0$ such that $N_\delta(A)\subset \Sigma$ and
 $\gamma(N_\delta(A))=\gamma(A)$, with $N_\delta(A)= \{x\in E| \text{dist}(x,A)\leq\delta\}$;
\item [(7)] The n-dimensional sphere $\mathbb{S}_n$ has a genus of $n$ by the Borsuk-Ulam Theorem.
 \end{itemize}
\end{proposition}

Clearly, the variational functional $J$ is not bounded from below in $X_r$. In fact, for
any $u\in X_r \backslash\{0\}$, arguing as the proof of Lemma \ref{geometry}-(ii), one has
\[
J(tu)\leq \frac{t^2}{2}\|u\|^2+ \lambda C t^6 \|u\|^6- \int_{\R^2}K(|x|)F(tu)dx\to
-\infty~ \text{as}~ t\to+\infty.
\]
Motivated by \cite{Azorero}, we introduce a truncated functional $\mathcal{J}$ (see \eqref{newfunctional} below)
which is bounded from below in $X_r$ and verifies that all critical points $u$ of $J$ with $J(u)<0$ are critical
points of $\mathcal{J}$.
To get around the obstacle caused by the general nonlinearity with critical exponential growth in equation \eqref{mainequation1},
we split the discussions
of any $u\in X_r$ into the following two cases.

\textbf{Case 1:} $\|u\|\leq \Upsilon$, where $\Upsilon>0$ comes from \eqref{fF}.\\
Let $\epsilon=1/4$ in \eqref{fF}, by
\eqref{g}, we have
\begin{equation}\label{Case1aa}
J(u)\geq \frac{1}{4}\|u\|^2-C_0\|u\|^{s}-\frac{\mu}{q}C\|u\|^q,
\end{equation}
where $C_0=C_0(\alpha,b_0,s)>0$ is a constant and $s>2$ is a constant given by \eqref{fF}. Set
\[
\sigma(t)\triangleq \frac{1}{4}t^2-C_0t^{s}-\frac{\mu}{q}Ct^q,~ \forall~t\geq0,
\]
then we derive
\begin{equation*}
 J(u)\geq \sigma(\|u\|)~ \text{for all}~  u\in X_r ~ \text{with}~ \|u\|\leq\Upsilon.
\end{equation*}
Since $q<2<s$, there exists a constant $\mu^{00}\in(0,\mu^0]$ such that $\sigma(t)$ possesses two unique zero points
$0<T_0(\mu)<T_1(\mu)$ for every $\mu\in(0,\mu^{00})$. We claim that $\lim_{\mu\to0^+}T_0(\mu)=0$. In fact,
it follows from $\sigma(T_i(\mu))=0$ for $i=0,1$ and $\sigma^\prime(T_0(\mu))>0>\sigma^\prime(T_1(\mu))$ that
\begin{equation}\label{Case1b}
 \frac{1}{4}T_i^2(\mu)-C_0T_i^{s}(\mu) -\frac{\mu}{q}CT_i^q(\mu)=0,
\end{equation}
and
\begin{equation}\label{Case1c}
\left\{
  \begin{array}{ll}
   \frac{1}{2}T_0^2(\mu) -s C_0T_0^{s}(\mu)- \mu CT_0^q(\mu) >0,\vspace{2mm}\\
   \frac{1}{2}T_1^2(\mu) -s C_0T_1^{s}(\mu) - \mu CT_1^q(\mu) <0.
  \end{array}
\right.
 \end{equation}
Eliminating the term ${\mu}CT_i^q(\mu)/q$ in the combing of
  \eqref{Case1b} and \eqref{Case1c}, we can obtain
\begin{equation}\label{mumu}
T_0(\mu)\leq \bigg[\frac{2-q}{4C_0(s-q)}\bigg]^{\frac{1}{ s-2  }}\leq T_1(\mu),
\end{equation}
which indicates that $T_0(\mu)$ is uniformly bounded with respect to $\mu$ since $C_0$ is independent of $\mu$.
Fix any sequence $\{\mu_n\}\subset (0,+\infty)$ with $\lim_{n\to\infty}\mu_n=0$ and suppose that
$T_0(\mu_n)\to T_0$ as $n\to\infty$. Letting $n\to\infty$ in \eqref{Case1b} and the first inequality in
 \eqref{Case1c}, respectively, we derive
\[
 \frac{1}{4}T_0^2-C_0T_0^{s}=0
~  \text{and} ~
 \frac{1}{2}T_0^2-s C_0T_0^{s}\geq0.
\]
which yields that
\[
\frac{2-s}{4}T_0^2\geq0.
\]
Because $s>2$ in \eqref{fF}, we derive $T_0=0$. By the arbitrariness of
$\{\mu_n\}$ with $\lim_{n\to\infty}\mu_n=0$, we can conclude that the claim is true.
Consequently, there exists a sufficiently small $\mu_4>0$ such that $T_0(\mu)<\Upsilon$
for each $\mu\in(0,\mu_4)$ and then
$T_0(\mu)<\min\{\Upsilon,T_1(\mu)\}$. Inspired by \cite{Azorero}, we can take a cut-off
function $\Psi(t)\in C_0^\infty(\R)$ satisfies $\Psi(t)\in[0,1]$ for any $t\geq0$ and
\[
\Psi(t)=
\left\{
\begin{array}{ll}
1, & \text{if}~ t\in[0,T_0(\mu)] \\
0, &\text{if}~ t\in[\min\{\Upsilon,T_1(\mu)\},+\infty).
\end{array}
\right.
\]
Then we define the following auxiliary functional
\begin{align}\label{newfunctional}
\mathcal{J}(u)&\triangleq\frac{1}{2}\|u\|^2+\frac{\lambda}{2}c(u)
 - \Psi(\|u\|) \int_{\R^2}K(|x|)F(u)dx
-\frac{\mu}{q}\int_{\R^2}g(|x|)|u|^qdx.
\end{align}
One can easily verify that $\mathcal{J}\in C^1(X_r,\R)$ and similar to \eqref{Case1aa},
\begin{equation}\label{Case1d}
\mathcal{J}(u)\geq \overline{\sigma}(\|u\|) ~\text{for all}~
 u\in X_r ~\text{with}~ \|u\|\leq\Upsilon,
\end{equation}
where $\overline{\sigma}:[0,+\infty)\to\R$ is defined by
\[
\overline{\sigma}(t)\triangleq\frac{1}{4}t^2-C_0\Psi(t)t^{s}-\frac{\mu}{q}Ct^q.
\]
Obviously, $\overline{\sigma}(t)\geq \sigma(t)$ for every $t\geq0$. By the definitions of
$J$ and $\mathcal{J}$, $J(u)=\mathcal{J}(u)$ for all $u\in X_r$ with $\|u\|\leq T_0(\mu)<\min\{\Upsilon,T_1(\mu)\}$.
Thereby, if $u\in X_r$ is a critical point of $\mathcal{J}$ with $\mathcal{J}(u)<0$ and $\|u\|\leq T_0(\mu)$, then
$u$ is also a critical point of $J$. To show that $\|u\|\leq T_0(\mu)$, it is necessary to make sure that $\mathcal{J}(u)\geq0$
for every $u\in X_r$ with $\|u\|\geq\Upsilon$. Next, we shall consider the other case.

\textbf{Case 2:} $\|u\|>\Upsilon$.\\
Notice that $\Psi(\|u\|)\equiv0$ in this case, then by \eqref{g}, we have
\begin{align*}
\mathcal{J}(u) &=\frac{1}{2}\|u\|^2+\frac{\lambda}{2}
\int_{\R^2}\frac{u^2}{|x|^2}\bigg(\int_{0}^{|x|}\frac{r}{2}u^2(r)dr\bigg)^2dx
-\frac{\mu}{q}\int_{\R^2}g(|x|)|u|^qdx   \\
  & \geq  \frac{1}{2}\|u\|^2-\frac{\mu}{q}C\|u\|^q \\
  &\triangleq \xi(\|u\|),
\end{align*}
where $\xi:[0,+\infty)\to\R$ is defined by
\[
\xi(t)=\frac{1}{2}t^2-\frac{\mu}{q}Ct^q.
\]
It is easy to compute that
\[
\min_{t\geq0}\xi(t)=\frac{q-2}{2q}\big(\mu C\big)^{\frac{2}{2-q}}<0.
\]
Obviously, $\xi(t)\geq 0$ if and only if $t\geq t_0\triangleq (2\mu C/q)^{1/(2-q)}$. So, it suffices to
chose $t_0\leq \Upsilon$ to ensure that $\mathcal{J}(u)\geq0$ for all $\|u\|\geq\Upsilon$, that is,
$\mu<q\Upsilon^{2-q}/(2C)\triangleq \mu_6$.

\begin{lemma}\label{criticalpoint}
There exists a constant $\mu_{**}^*>0$ such that for
any $\mu\in(0,\mu_{**}^*)$, we have the following results:\\
\emph{(i)} if $\mathcal{J}(u)<0$, then $\|u\|<T_0(\mu)$ and $\mathcal{J}(v)=J(v)$ for $v$ in a small
neighborhood of $u$;\\
\emph{(ii)} $\mathcal{J}$ satisfies a local $(PS)_d$ condition for all $d<0$.
\end{lemma}

\begin{proof}
It is easy to see that
\[
\sigma(t)= t^q\bigg
(\frac{1}{4}t^{2-q}-C_0t^{s-q}-\frac{\mu}{q}C\bigg)=0,~ \forall~t\geq0,
\]
has two unique nonzero roots for any
\[
0<\mu<\mu^{00}\triangleq\min\bigg\{\mu^0, \frac{q(s-2)}{4(s-q)C}
\bigg(\frac{2-q}{4C_0(s-q)}\bigg)^{\frac{2-q}{s-2}}   \bigg\},
\]
where $\mu^0>0$ is defined by Lemma \ref{secondsolution}, $s>2$ is a constant given by \eqref{fF},
$1\leq q<2$, $C_0=C_0(\alpha,b_0,s)>0$ and $C>0$ are constants. Because we have
verified that $\lim_{\mu\to0^+}T_0(\mu)=0$, there exists a constant $\mu_4>0$ such that $T_0(\mu)<\min\{\Upsilon,T_1(\mu)\}$ for all
 $\mu\in(0,\mu_4)$. In view of the Case 2 and \eqref{mumu}, we can deduce that $T_1(\mu)>t_0$ for any
\[
0<\mu<\mu_5\triangleq \frac{q}{2C}\bigg[\frac{2-q}{4C_0(s-q)}\bigg]^{\frac{2-q}{ s-2}}.
\]
The Case 2 indicates that $\mathcal{J}(u)\geq0$ for every $\|u\|\geq\Upsilon$ whenever
$0<\mu<q\Upsilon^{2-q}/(2C)\triangleq \mu_6$. Set $\mu_{**}^*=\min\{\mu^{00},\mu_4,\mu_5,\mu_6\}>0$,
then all the conclusions in the above two cases are true for any $\lambda>0$ and $\mu\in(0,\mu_{**}^*)$. Next,
 we give the proof of the lemma.

(i) For all $\mu\in(0,\mu_{**}^*)$, $\mathcal{J}(u)< 0$ implies that $\|u\|< \Upsilon$.
It follows from \eqref{Case1d} that $\sigma(\|u\|)\leq \overline{\sigma}(\|u\|)\leq \mathcal{J}(u)<0$.
The definition of $\sigma(t)$ reveals us that
either $\|u\|<T_0(\mu)$ or $T_1(\mu)<\|u\|<\Upsilon$, because $T_1(\mu)\geq \Upsilon$ immediately yields that $\|u\|<T_0(\mu)$.
Arguing it indirectly and we can suppose that $T_1(\mu)<\|u\|<\Upsilon$, then we get that $\mathcal{J}(u)\geq
\xi(\|u\|)\geq0$ since $\|u\|>T_1(\mu)>t_0$, a contradiction. Therefore, we derive
$\|u\|<T_0(\mu)$ and $\mathcal{J}(v)=J(v)$ for any $v\in X_r$ satisfying $\|v-u\|<T_0(\mu)-\|u\|$.

(ii) For every $\mu\in(0,\mu_{**}^*)$, let $\{u_n\}\subset X_r$ be any sequence such that
$\mathcal{J}(u_n)\to d<0$ and $\mathcal{J}^{\prime}(u_n)\to 0$. Therefore, for sufficiently large $n\in \mathbb{N}$,
one gets $J(u_n)=\mathcal{J}(u_n)\to d<0$ and $J^{\prime}(u_n)=\mathcal{J}^{\prime}(u_n)\to 0$. By
means of a similar argument in Lemma \ref{secondsolution},
$\{u_n\}$ has a strongly convergent subsequence. The proof is complete.
\end{proof}

In order to construct the suitable minimax sequence of negative critical values for $\mathcal{J}$, we need a finite dimensional
subsequence of $X_r$. Since $X_r$ is a separable and reflexive Hilbert space, there exists an orthogonal basis
$\{e_i\}_{i=1}^\infty$ for $X_r$. Hence, for every
$n\in \mathbb{N}$, we can set
$E_n\triangleq\text{span}\{e_1,e_2,\cdots,e_n\}$ and $Z_n\triangleq\oplus_{i=1}^nE_n$.
On the other hand, for any $\epsilon>0$, we define
\[
\mathcal{J}^{-\epsilon}\triangleq\{u\in X_r|  \mathcal{J}(u)\leq -\epsilon\}.
\]

\begin{lemma}\label{gamman}
For any $\mu>0$ and $n\in \mathbb{N}$, there exists $\epsilon_n>0$ such that
$\gamma(\mathcal{J}^{-\epsilon_n})\geq n$.
\end{lemma}

\begin{proof}
Fix $\mu>0$ and $n\in \mathbb{N}$. Since $\dim Z_n<+\infty$, there exists a positive
constant $c(n)>0$ such that
\[
c(n)\|u\|^q\leq \int_{\R^2}g(|x|)|u|^qdx,~ \forall u\in Z_n.
\]
Thereby, for any $u\in Z_n$ with $\|u\|<T_0(\mu)$, by Lemma \ref{imbedding3} and $(f_1)$, there holds
\[
 \mathcal{J}(u)\leq  \frac{1}{2}\|u\|^2+C
\|u\|^6-\frac{\mu}{q}c(n)\|u\|^q.
\]
Because $q<2$, we can choose a sufficiently small $r_n\in (0,T_0(\mu))$  and a constant $\epsilon_n>0$
such that $\mathcal{J}(u)\leq -\epsilon_n<0$ for every $u\in Z_n$ with $\|u\|=r_n$.
Let $\mathbb{S}_{r_n}\triangleq\{u\in Z_n|\|u\|=r_n\}$, then it follows from Proposition \ref{gamma}-(7) that
$\gamma(\mathcal{J}^{-\epsilon_n})\geq \gamma(\mathbb{S}_{r_n})=n$. The proof is complete.
\end{proof}

For any $n\in \mathbb{N}$, define
\[
\Sigma_n=\big\{A\in\Sigma|  \gamma(A)\geq n\big\},
\]
and
\[
c_n=\inf_{A\in\Sigma_n}\sup_{u\in A}\mathcal{J}(u).
\]
Before proving the main results, we state some crucial properties of $\{c_n\}_{n\in \mathbb{N}}$.

\begin{lemma}\label{cncn}
For any $\mu>0$ and $n\in \mathbb{N}$, then
\[
-\infty<c_n\leq -\epsilon_n<0,~~~\forall n\in \mathbb{N}.
\]
Moreover, all $c_n$ are critical values of $\mathcal{J}$ and $\lim_{n\to\infty}c_n=0$ if $\mu\in(0,\mu_{**})$.
\end{lemma}

\begin{proof}
Fix $\mu>0$ and $n\in \mathbb{N}$. According to Lemma \ref{gamman}, there is
a constant $\epsilon>0$ such that $\gamma(\mathcal{J}^{-\epsilon_n})\geq n$ and
thus $\mathcal{J}^{-\epsilon_n}\in\Sigma_n$ since $\mathcal{J}$ is continuous and even. From $\mathcal{J}(0)=0$, one
has $0\not\in\mathcal{J}^{-\epsilon_n}$. Furthermore,
$\sup_{u\in \mathcal{J}^{-\epsilon_n}} \mathcal{J}(u)\leq-\epsilon_n$. Since $\mathcal{J}$ is bounded from below in
$X_r$,
\[
-\infty<c_n=\inf_{A\in\Sigma_n}\sup_{u\in A}\mathcal{J}(u)\leq \sup_{u\in \mathcal{J}^{-\epsilon}} \mathcal{J}(u)\leq-\epsilon_n<0.
\]
It follows from Lemma \ref{criticalpoint}-(ii) that all $c_n$ are critical values of $\mathcal{J}$. It's obvious that
$c_n\leq c_{n+1}$ for every $n\in \mathbb{N}$, there is a constant $\overline{c}\leq0$ such that $\lim_{n\to\infty}c_n=\sup_{n\in
\mathbb{N}}c_n\triangleq\overline{c}$.  Arguing by contradiction, we suppose that
$\overline{c}<0$. So, $K_{\overline{c}}$
is compact by Lemma \ref{criticalpoint}-(ii), where
\[
K_{\overline{c}}\triangleq\{u\in X_r|\mathcal{J}^\prime(u)=0,~
\mathcal{J}(u)=\overline{c}\}.
\]
In view of Proposition \ref{gamma}-(6), $\gamma(K_{\overline{c}})\triangleq\overline{n}<\infty$ and there exists a constant
$\delta>0$ such that $\gamma(K_{\overline{c}})=\gamma(N_\delta(K_{\overline{c}}))=\overline{n}$.

From the deformation lemma (see \cite[Theorem A.4]{Rabinowitz}), there exist
a constant $\epsilon\in(0,-\overline{c})$ and an
homeomorphism $\eta:X_r \to X_r $ such that
\begin{equation}\label{cncn1}
   \eta\big(\mathcal{J}^{\overline{c}+\epsilon}\backslash N_\delta(K_{\overline{c}})\big)\subset \mathcal{J}^{\overline{c}-\epsilon}.
\end{equation}
Since $\overline{c}=\sup_{n\in \mathbb{N}}c_n$, there exists $n\in \mathbb{N}$ and $c_n>\overline{c}-\epsilon$ and
$c_{n+\overline{n}}\leq \overline{c}$. By the definition of $c_{n+\overline{n}}$, there exists $A\in\Gamma_{n+n_0}$
such that $\sup_{u\in A}\mathcal{J}(u)\leq \overline{c}+\epsilon$. By Proposition \ref{gamma}-(2)(3),
\begin{equation}\label{cncn2}
\gamma\big(\eta( \overline{A\backslash N_\delta(K_{\overline{c}})})\big)=\gamma\big(\overline{A\backslash N_\delta
(K_{\overline{c}})}\big)\geq \gamma(A)-\gamma(N_\delta(K_{\overline{c}}))\geq n,
\end{equation}
which yields that
\begin{equation}\label{cncn3}
\sup_{u\in \eta( \overline{A\backslash N_\delta(K_{\overline{c}})})}\mathcal{J}(u)\geq c_n>\overline{c}-\epsilon.
\end{equation}
Combing \eqref{cncn1} and \eqref{cncn2}, we have
\[
\eta( \overline{A\backslash N_\delta(K_{\overline{c}})})\subset \eta\big(\mathcal{J}^{\overline{c}+\epsilon}\backslash N_\delta(K_{\overline{c}})\big)
\subset \mathcal{J}^{\overline{c}-\epsilon},
\]
which is a contradiction to \eqref{cncn3}. So, $\overline{c}=0$ and we can finish the proof of this lemma.
\end{proof}

Now, we can present the proof of Theorem \ref{maintheorem3}.

\begin{proof}[\textbf{\emph{Proof of Theorem \ref{maintheorem3}}}]
Choosing $\mu_{**}^*>0$ as Lemma \ref{criticalpoint}, then the conclusions in Lemmas \ref{criticalpoint} and \ref{cncn}
remain true for every $\mu\in(0,\mu_{**}^*)$. It is easy to conclude that all the assumptions of the new version of
symmetric mountain-pass theorem due to Kajikiya \cite{Kajikiya} are satisfied. Consequently,
equation \eqref{mainequation1} possesses
infinitely many solutions for every $\lambda>0$ and $\mu\in(0,\mu_{**}^*)$. The proof is complete.
\end{proof}

\end{document}